\documentclass[11pt]{amsart}
\usepackage{amssymb, amsmath,a4wide}

\newtheorem{thm}{Theorem}[section]
\newtheorem{crl}[thm]{Corollary}

\newtheorem{lmm}[thm]{Lemma}

\newtheorem{prp}[thm]{Proposition}

\theoremstyle{definition}

\theoremstyle{remark}
\newtheorem*{rem}{Remark}

\newcommand{\R}{\mathbb{R}}
\newcommand{\C}{\mathbb{C}}
\newcommand{\N}{\mathbb{N}}
\newcommand{\Z}{\mathbb{Z}}
\newcommand{\T}{\mathbb{T}}
\newcommand{\la}{\langle}
\newcommand{\ra}{\rangle}
\newcommand{\pd}{\partial}

\newcommand{\mL}{\mathcal{L}}
\newcommand{\eps}{\varepsilon}
\newcommand{\tv}{\tilde{v}}
\newcommand{\bM}{\mathbb{M}}
\DeclareMathOperator{\sech}{sech}

\DeclareMathOperator{\spann}{span}

\begin{document}
\title[$L^2$-stability of solitary waves for the KdV equation]{$L^2$-stability of solitary waves for the KdV equation via Pego and Weinstein's method}
\author[T.~Mizumachi]{Tetsu Mizumachi}
\address{Faculty of Mathematics, Kyushu University, Fukuoka 819-0395, Japan.}
\email{mizumati@math.kyushu-u.ac.jp}
\author[N.~Tzvetkov]{Nikolay Tzvetkov}
\address{University of Cergy-Pontoise, UMR CNRS 8088, Cergy-Pontoise, F-95000 and Institut Universitaire de France}
\email{nikolay.tzvetkov@u-cergy.fr}
\thanks{This research is Supported by Grant-in-Aid for Scientific Research (No. 21540220).}  
\begin{abstract}      
In this article, we will prove $L^2(\R)$-stability of $1$-solitons for
the KdV equation by using exponential stability property of the
semigroup generated by the linearized operator. The proof follows
the lines of recent stability argument of Mizumachi (\cite{M1})
and Mizumachi, Pego and Quintero (\cite{MPQ})
which show stability in the energy class by using strong linear stability
of solitary waves in exponentially weighted spaces.
\par
This gives an alternative proof of Merle and Vega (\cite{MV}) which shows
$L^2(\R)$-stability of $1$-solitons for the KdV equation 
by using the Miura transformation. 
Our argument is a refinement of Pego and Weinstein (\cite{PW}) that proves
asymptotic stability of solitary waves in
exponentially weighted spaces. We slightly improve the $H^1$-stability of
the modified KdV equation as well.
\end{abstract}
\maketitle

\section{Introduction}
\label{sec:intro}
In this article, we discuss stability of solitary waves for
the generalized KdV equations
\begin{equation}
  \label{eq:gKdV}
  \pd_tu+\pd_x^3u+3\pd_x(u^p)=0\quad \text{for $(t,x)\in\R_+\times \R$.}
\end{equation}
The case $p=2$ corresponds to the KdV equation and
describes a motion of shallow water waves.
The case $p=3$ corresponds to the modified KdV equation.
The generalized KdV equations have a family of solitary wave solutions 
$\{\varphi_c(x-ct+\gamma)\mid c>0\,,\,\gamma\in\R\}$,
where
\begin{equation}
  \label{eq:soliton}
\varphi_c(x)=\alpha_c\sech^{2/(p-1)}(\beta_cx)\,,\quad
\alpha_c=\left(\frac{(p+1)c}{6}\right)^{1/(p-1)}\,,\quad
\beta_c=\frac{p-1}{2}\sqrt{c}\,,
\end{equation}
and $\varphi_c$ is a solution of
\begin{equation}
  \label{eq:B}
\varphi_c''-c\varphi_c+3\varphi_c^p=0\quad\text{for $x\in\R$.}
\end{equation}
Solitary waves play an important role among the solutions of \eqref{eq:gKdV}.
Indeed, solutions of the KdV equation resolve into a train of solitary
waves and an oscillating tail if the initial data are 
rapidly decreasing functions (see \cite{Sc}).
\par
Stability of solitary waves has been studied for many years since
Benjamin (\cite{Benjamin}) and Bona (\cite{Bona}).
Let us briefly introduce their result by Weinstein's argument (\cite{BSS,We}).
Eq.~\eqref{eq:gKdV} has conserved quantities
\begin{gather*}
\int_{\R}u^2(x)\,dx\quad\text{(the momentum),}\\
E(u)=\int_\R\left(\frac12(\pd_xu)^2(x)-\frac{3}{p+1}u^{p+1}(x)\right)\,dx\quad
\text{(the Hamiltonian).} 
\end{gather*}
Let $\mathcal{M}_c:=\{u\in H^1(\R)\mid \|u\|_{L^2}=\|\varphi_c\|_{L^2}\}$.
The set $\mathcal{M}_c$ is invariant under the flow generated
by \eqref{eq:gKdV} and the fact that $\varphi_c$ minimizes
$E|_{\mathcal{M}_c}$ for $p=2$, $3$ and $4$ implies orbital stability
of $\varphi_c$,
that is, for any $\eps>0$, there exists a $\delta>0$
such that if $u(0,x)=\varphi_c(x)+v_0(x)$ and $\|v_0\|_{H^1}<\delta$, then
$$\sup_{t\in\R}\inf_{\gamma\in\R}\|u(t,\cdot+\gamma)-\varphi_c\|_{H^1}<\eps\,.$$
\par

In order to study blow up problems of \eqref{eq:gKdV} with $p=5$,
Martel and Merle (\cite{MM1}) established a Liouville theorem for solutions
around solitary wave solutions of \eqref{eq:gKdV}.
Using the Liouville theorem, they prove that
solitary wave solutions are asymptotic stabile in $H^1_{loc}(\R)$
if $p=2$, $3$ and $4$.
Later, they gave a more direct proof by using a time global virial
estimate around solitary wave solutions (\cite{MM3}).
We refer \cite{MM4} for recent developments
such as inelastic collision of solitary waves for \eqref{eq:gKdV}
with $p=4$. 
\par

The $L^2(\R)$-stability of solitary wave solutions was first studied
by Merle and Vega (\cite{MV}) for the KdV equation by using the Miura
transformation and the fact that kink solutions of the defocusing mKdV
equation is stable to perturbations in $H^1(\R)$.  Indeed, a
combination of the Miura transformation and the Galilean
transformation
$$u(t,x)=M^+(v)(t,x-3ct)+\frac{c}{2}\,,\quad M^+(v)=\pd_xv-v^2\,,$$
is isomorphic between an $L^2$-neighborhood of a $1$-soliton $\varphi_c(x-ct)$
and an $H^1(\R)\times \R$-neighborhood of $(\psi_c(x+ct),c)$, where
$\psi_c=\sqrt{c/2}\tanh(\sqrt{c/2}x)$ and $\psi_c(x+ct)$ is a kink solution of
the defocusing mKdV $\pd_tv+\pd_x^3v-2\pd_x(v^3)=0$.
\par
Their result has been extended to prove $L^2(\R_x\times\T_y)$-stability of
line solitons for the KP-II equation (\cite{MT}),
$L^2(\R)$-stability of $1$-solitons for the $1d$-cubic NLS
(\cite{MPel}) and $L^2(\R)$-stability of $N$-solitons for KdV 
(\cite{AM}) and the structural stability of $1$-solitons for KdV
in $H^{-1}(\R)$ (\cite{Buck-Koch}).
These results rely on the B\"acklund transformations
which are peculiar to the integrable systems.
In this article, we will show $L^2(\R)$-stability of KdV $1$-solitons
without using the integrability of the KdV equation.
\par

It is common for the long wave models that the main solitary wave
moves faster than the other parts of the solution, which leads to
strong linear stability of solitary waves in exponentially weighted
spaces (see e.g.~\cite{MW,PS,PW,PW2}). This property was first used by
Pego and Weinstein (\cite{PW}) to prove asymptotic stability of solitary
waves of the generalized KdV equations
to exponentially localized perturbations.
Their argument turns out to be useful especially when solitary waves cannot be
characterized as (constrained) minimizers of conserved quantities.
Applying the idea of \cite{PW}, Friesecke and Pego (\cite{FP1,FP2,FP3,FP4})
proved that solitary waves to the FPU lattices are stable for exponentially
localized perturbations (see also \cite{MP}).
 Mizumachi (\cite{M1,M2}) extended \cite{FP2}
and prove stability of $N$-soliton like solutions
in the energy class by suitability decomposing the remainder part of
the solution into a sum of small waves which moves much slower than
the main waves and exponentially localized parts.
The argument has been applied to the Benney-Luke equation which is one of
bidirectional models of the water waves whose solitary waves
in the weak surface tension regime are infinitely indefinite saddle point of
the energy-momentum functional (\cite{MPQ}).
Recently, Mizumachi (\cite{M3}) has proved transversal stability of line
solitons for the KP-II equation in exponentially weighted space.
We expect the argument used in \cite{M1,MPQ} is useful to prove
stability of line solitons for the KP-II equation in unweighted spaces.
In this article, we will apply the argument used in \cite{M1,MPQ}
to the KdV equation and give an alternative
proof of the following result by Merle and Vega (\cite{MV}).
\begin{thm}[\cite{MV}]
  \label{thm:1}
Let $p=2$ and $c_0>\sigma>0$.
Then there exist positive constants $C$ and $\delta$
satisfying the following.
Suppose that $u(t,x)$ is a solution of \eqref{eq:gKdV} satisfying
$u(0,x)=\varphi_{c_0}(x)+v_0(x)$ and $\|v_0\|_{L^2}<\delta$. Then
there exist $c_+>0$ and a $C^1$-function $x(t)$ such that
\begin{gather}
  \label{eq:os-2}
\sup_{t\ge0} \left\|u(t,\cdot)-\varphi_{c_0}(\cdot-x(t))\right\|_{L^2}
\le C\|v_0\|_{L^2}^{1/2}\,,\\
\label{dotx}
c_+=\lim_{t\to\infty}\dot{x}(t)\,,\\
\label{c-c+}
|c_+-c_0|+\sup_{t\ge0}|\dot{x}(t)-c_0|\le C\|v_0\|_{L^2}\,,\\
  \label{eq:as-2}
\lim_{t\to\infty} \left\|u(t,\cdot)-\varphi_{c_+}(\cdot-x(t))
\right\|_{L^2(x\ge \sigma t)}=0\,.
\end{gather}
\end{thm}
\begin{rem}
  The $L^2(\R)$ well-posedness of the KdV equation was proved by
Bourgain (\cite{Bourgain}).
\end{rem}
\begin{rem}
We expect that stability argument of $N$-solitary wave solutions to the
FPU lattices (\cite{M2}) is applicable to $N$-solitary wave solutions
of the long wave models as well.  
\end{rem}

For the mKdV equation, 
we slightly improve orbital stability of $1$-solitons in $H^1(\R)$.
Note that the mKdV equation is well-posed in $H^s(\R)$
with $s\ge 1/4$. See \cite{KPV1,KPV2}.
\begin{thm}
  \label{thm:2}
Let $p=3$ and $c_0>\sigma>0$.
Then there exist positive constants $C$ and $\delta$
satisfying the following.
Suppose that $u(t,x)$ is a solution of \eqref{eq:gKdV} satisfying
$u(0,x)=\varphi_{c_0}(x)+v_0(x)$ and $\|v_0\|_{L^2}^{3/4}\|v_0\|_{H^1}^{1/4}<\delta$.
Then there exist $c_+>0$ and a $C^1$-function $x(t)$ such that
\begin{gather}
  \label{eq:os-3}
\sup_{t\ge0} \left\|u(t,\cdot)-\varphi_{c_0}(\cdot-x(t))\right\|_{L^2}
\le C\|v_0\|_{L^2}^{1/2}\,,\\
c_+=\lim_{t\to\infty}\dot{x}(t)\,,\\
\label{c-c+2}
|c_+-c_0|+\sup_{t\ge0}|\dot{x}(t)-c_0|\le C\|v_0\|_{L^2}\,,\\
  \label{eq:as-3}
\lim_{t\to\infty} \left\|u(t,\cdot)-\varphi_{c_+}(\cdot-x(t))
\right\|_{L^2(x\ge \sigma t)}=0\,.
\end{gather}
\end{thm}

Finally, let us introduce several notations. 
Let $L^p_a$ ($1\le p\le \infty$) and $H^k_a$ ($k\in\N$) be exponentially
weighted spaces, writing
$$L^p_a=\{g\mid e^{ax}g\in L^p(\R)\}\,,\quad
H^k_a=\{g \mid \pd_x^jg\in L^2_a\text{ for $0\le j\le k$}\}\,,$$
with norms
\begin{equation*}
\|g\|_{L^p_a}=\|e^{a\cdot}g\|_{L^p(\R)}\,,\quad
\|g\|_{H^k_a}=\left(\sum_{0\le j\le k}\|\pd_x^jg\|_{L^2_a}^2\right)^{1/2}\,.
\end{equation*}
We define  $\la\cdot,\cdot\ra$ and $\la\cdot,\cdot\ra_{t,x}$ as
$$\la u_1,u_2\ra=\int_\R u_1(x)u_2(x)\,dx\,,\quad
\la v_1,v_2\ra_{t,x}=\int_\R \int_\R v_1(t,x)v_2(t,x)\,dxdt\,.$$
For any Banach spaces $X$, $Y$, we denote by $B(X,Y)$ the space of
bounded linear operators from $X$ to $Y$. We abbreviate
$B(X,X)$ as $B(X)$.
We use $a\lesssim b$ and $a=O(b)$ to mean that there exists a
positive constant such that $a\le Cb$. 
Various constants will be simply denoted
by $C$ and $C_i$ ($i\in\mathbb{N}$) in the course of the
calculations.


\section{Linear stability of $1$-solitons}
\label{sec:linear}
In this section, we recall strong linear stability 
of $1$-soliton solutions in the exponentially weighted space $L^2_a$
for the sake of self-containedness.
Let
$$u(t,x)=\varphi_c(y)+v(t,y)\,,\quad y=x-ct$$
and linearize nonlinear terms of \eqref{eq:gKdV} around $v=0$. Then we have
\begin{equation}
  \label{eq:LKdV-1}
  \pd_tv+\mL_{c}v=0\,,
\end{equation}
where $\mL_c=\pd_y(\pd_y^2-c+f'(\varphi_c))$ with $f(u)=3u^p$.
The linearized operator $\mL_c$ has a generalized kernel associated with the
infinitesimal changes in the location and the speed of the solitary
waves.  Let $\xi^1_c(y)=\pd_y\varphi_c(y)$, $\xi^2_c(y)=\pd_c\varphi_c(y)$ and
\begin{equation}
  \label{eq:defzetas}
\zeta^1_c(y)=
-\theta_1(c)\int_{-\infty}^y\pd_c\varphi_c(y_1)\,dy_1+\theta_2(c)\varphi_c(y)
\,,\quad \zeta_c^2(y)=\theta_1(c)\varphi_c(y)\,,
\end{equation}
where $\theta_1(c)=1/\int_\R \varphi_c(y)\pd_c\varphi_c(y)\,dy$
and $\theta_2(c)=\theta_1(c)^2\left(\int_\R\pd_c\varphi_c(y)\,dy\right)^2/2$.
Differentiating \eqref{eq:B} twice with respect to $y$ and $c$, we have
\begin{equation}
  \label{eq:xis}
\mL_c\xi_c^1=0\,,\quad  \mL_c\xi_c^2=\xi_c^1\,.  
\end{equation}
In view of \eqref{eq:xis} and the fact that (formally)
$\pd_y\mL_c^*=-\mL_c\pd_y$,
\begin{equation}
  \label{eq:etas}
\mL_c^*\zeta_c^1=\zeta_c^2\,,\quad  \mL_c^*\zeta_c^2=0\,.
\end{equation}
For $i$, $j$=$1$, $2$,
$$\int_\R\xi_c^i(y)\zeta_c^j(y)\,dy=\delta_{ij}\,.$$
\par
Since $d\|\varphi_c\|_{L^2}^2/dc\ne0$ for $p\ne 5$,
the algebraic multiplicity of the eigenvalue $0$ is two if
 $p\ne 5$.
Let  $P_c:L^2_a\to L^2_a$ ($0<a<\sqrt{c}$) be the spectral projection
associated with the generalized kernel of $\mL_c$ and let $Q_c=I-P_c$.
Then 
\begin{gather*}
P_cv= \la v\,,\, \zeta_c^1\ra \xi_c^1+\la v\,,\, \zeta_c^2\ra \xi_c^2\,,\\
\operatorname{Range}(\mL_cP_c) \subset \operatorname{Range}(P_c)\,,\quad
\operatorname{Range}(\mL_cQ_c) \subset \operatorname{Range}(Q_c)\,.
\end{gather*}
\par
Next we recall the spectrum of the linearized operator $\mL_c$.  The
spectrum of $\mL_c$ in $L^2(\R)$ consists of $i\R$ (see \cite{PW0}).
However if $0<a<\sqrt{c}$, the essential spectrum of
$\mL_c$ in $L^2_a$ locates in the stable half plane.
Indeed, the spectrum of $\mL_c$ in $L^2_a$ is equivalent to the spectrum of
$e^{a\cdot}\mL_ce^{-a\cdot}$ in $L^2(\R)$ and 
by Weyl's essential spectrum theorem, 
$$\sigma_{ess}(\mL_c)=\{ip(\xi+ia)\mid \xi\in\R\}
\,,\quad p(\xi)=-(\xi^3+c\xi)\,.$$
Since 
\begin{equation}
  \label{eq:pia}
p(\xi+ia)=-\{\xi^3+(c-3a^2)\xi\}-i\{a(c-a^2)+3a\xi^2\}\,,  
\end{equation}
we have  $\sigma_{ess}(\mL_c)\subset
\{\lambda\in\C\mid \Re\lambda\ge a(c-a^2)>0\}$ if $a\in(0,\sqrt{c})$.

The complement of $\sigma_{ess}(\mL_c)$
consists of two connected components.
We denote by $\Omega(a)$ one of these components which contains
the unstable half plane. Pego and Weinstein 
\cite[Proposition~2.6, Theorem~3.1]{PW} prove spectral stability
of $\mL_c$ in the exponentially weighted space $L^2_a$.
\begin{prp}[Spectral stability of 1-solitons (\cite{PW})]
  \label{spectral-stability}
Let $p=2$ or $3$. Suppose $c>0$ and $a\in(0,\sqrt{c})$.
Then the operator $\mL_c$ in $L^2_a$ has no eigenvalue
in $\Omega(a)$ other than $0$ whose algebraic multiplicity is two
and 
$\sigma_{ess}(\mL_c)\subset\{\lambda\in\C\mid \Re\lambda\ge a(c-a^2)>0\}$.
\end{prp}
\par
Let $R(\lambda)=(i\lambda+\mL_c)^{-1}$ and
$\Sigma_b:=\{\lambda\in\C\mid \Im \lambda<b\}$.
The spectral stability of $\mL_c$ implies that
\begin{equation}
  \label{eq:rbound0}
\sup_{\lambda\in \Sigma_b} \left\|R(\lambda)Q_c\right\|_{B(L^2_a)}<\infty
\quad\text{for $b$ satisfying $0<b<a(c-a^2)$.}
\end{equation}
Let
$\mL_0=\pd_y^3-c\pd_y$, $R_0(\lambda)=(i\lambda+\mL_0)^{-1}$ and
$V=\pd_y(f'(\varphi_c)\cdot)$.
By Plancherel's theorem and \eqref{eq:pia},
\begin{equation}
\label{eq:resolvent}
\begin{split}
\|R_0(\lambda)\|_{B(L^2_a,H^k_a)}
 \lesssim & \sup_{\xi\in\R}
\frac{1+|\xi+ia|^k}{|\lambda+p(\xi+ia)|}
\\\lesssim & (1+|\lambda|)^{-(2-k)/3}
\quad\text{for $\lambda\in\Sigma_b$ and $0\le k\le2$.}  
\end{split}
\end{equation}
In view of \eqref{eq:resolvent} and the fact that $f'(\varphi_c)$
is exponentially localized, we see that  $VR_0(\lambda)$ is compact on $L^2_a$
and that $I+VR_0(\lambda)$ has a bounded inverse unless $\lambda$ is an
eigenvalue of $\mL_c$. Hence it follows from
Proposition~\ref{spectral-stability} that
$\|(I+VR_0(\lambda))^{-1}Q_c\|_{B(L^2_a)}$ is bounded on any compact
subset of $\Sigma_b$. Moreover, Eq.~\eqref{eq:resolvent} with $k=1$ implies that
$\|VR_0(\lambda)\|_{B(L^2_a)}\le \frac12$ for large $\lambda$.
Thus we have
\begin{equation}
\label{eq:r-bound}
\sup_{\lambda\in \Sigma_b} \left\|R(\lambda)Q_c\right\|_{B(L^2_a;H^2_a)}
\lesssim  \sup_{\lambda\in\Sigma_b}\|R_0(\lambda)\|_{B(L^2_a;H^2_a)}<\infty
\end{equation} 
since $R(\lambda)=R_0(\lambda)(I+VR_0(\lambda))^{-1}$.
\par
Once \eqref{eq:rbound0} is established,
the Gearhart-Pr\"uss theorem (\cite{Gh,Pruss}) on $C_0$-semigroups
on Hilbert spaces implies exponential linear stability of $e^{-t\mL_c}Q_c$.
\begin{prp}[Linear stability of 1-solitons (\cite{PW})]
\label{prop:linear}
Let $p=2$ or $3$, $c>0$ and $a\in(0,\sqrt{c})$.
Then there exist positive constants $K$ and $b$ such that for every 
$f\in L^2_a$ and $t\ge0$,
\begin{equation}
  \label{eq:expdecay-1}
\|e^{-t\mL_c}Q_cf\|_{L^2_a}\le Ke^{-bt} \|f\|_{L^2_a}\,.
\end{equation}
\end{prp}
Exponential stability of $e^{-t\mL_c}Q_c$ reflects that the largest
solitary wave moves faster to the right than any other waves.
\par
Kato \cite{K} tells us that $e^{-t\pd_x^3}$ has a strong smoothing effect
on $L^2_a$. We will use the property to deal with nonlinear terms.
\begin{crl}
  \label{cor:s-smoothing}
Let $p=2$ or $3$, $c>0$ and $0<a<\sqrt{c}$.
Then there exist positive constants $K'$ and $b'$ such that
\begin{align}
  \label{eq:smoothing-a}
& \|e^{-t\mL_c}Q_c\pd_y^j f\|_{L^2_a}\le K'e^{-b't}t^{-(2j+1)/4} \|f\|_{L^1_a}
\quad\text{for $j=0$, $1$, $f\in L^1_a$ and $t>0$,}\\
\label{eq:smoothing-b}
& \|e^{-t\mL_c}Q_c\pd_y f\|_{L^2_a}\le K'e^{-b't}t^{-1/2} \|f\|_{L^2_a}
\quad\text{for $f\in L^2_a$ and $t>0$.}
\end{align}
\end{crl}

\begin{proof}
By the definition of the Fourier transform and Plancherel's theorem,
\begin{equation}
  \label{eq:plancherel}
 \|f\|_{L^2_a}=\|\hat{f}(\cdot+ia)\|_{L^2}\quad\text{for $f\in L^2_a$.} 
\end{equation}
By \eqref{eq:plancherel} and \eqref{eq:pia}, we have for $j\in \Z_{\ge0}$,
\begin{align*}
\|\pd_y^je^{-t\mL_0}f\|_{L^2_a}=& \left(\int_\R\left|
(\xi+ia)^je^{-itp(\xi+ia)}\hat{f}(\xi+ia)\right|^2\,d\xi\right)^{1/2}
\\ \le & e^{-a(c-a^2)t}
\left(\int_\R\left|
(\xi+ia)^je^{-3at\xi^2}\hat{f}(\xi+ia)\right|^2\,d\xi\right)^{1/2}\,.
\end{align*}
Since $\|\xi^ke^{-3at\xi^2}\|_{L^2}\lesssim  t^ {-(2k+1)/4}$ for $k\ge0$ 
and $\|\hat{f}(\cdot+ia)\|_{L^\infty}\lesssim \|f\|_{L^1_a}$,
\begin{equation}
\label{eq:smoothing0}
\|\pd_y^je^{-t\mL_0}f\|_{L^2_a}\lesssim e^{-a(c-a^2)t}(1+t^{-(2j+1)/4})
\|f\|_{L^1_a}\quad\text{for $j\in \Z_{\ge0}$.}
\end{equation}
Similarly, we have
\begin{equation}
\label{eq:smoothing1}
\|\pd_y^je^{-\mL_0}f\|_{L^2_a} \lesssim e^{-a(c-a^2)t}(1+t^{-j/2})\|f\|_{L^2_a}
\quad\text{for $j\in \Z_{\ge0}$}\,.
\end{equation}
\par
Now we will show \eqref{eq:smoothing-a} with $j=0$.
Let $v(t)$ be a mild solution of \eqref{eq:LKdV-1} with
$v(0)=Q_cf$ and $f\in L^1_a$, that is,
\begin{equation}
  \label{eq:2}
v(t)=e^{-t\mL_0}Q_cf
-\int_0^t e^{-(t-s)\mL_0}\pd_y\left(f'(\varphi_c)v(s)\right)\,ds
=:Tv(t)\,.  
\end{equation}
Since $Q_c\in B(L^1_a)$ and $f'(\varphi_c)$ is bounded,
it follows from \eqref{eq:smoothing0} and \eqref{eq:smoothing1} that
$$
\|Tv\|_{L^2_a} \lesssim 
e^{-a(c-a^2)t}t^{-1/4}\|f\|_{L^1_a}
+\int_0^t e^{-a(c-a^2)(t-s)}(t-s)^{-1/2}\|v(s)\|_{L^2_a}\,ds\,.
$$
By the contraction mapping theorem, there exists a unique solution of
\eqref{eq:2} satisfying
\begin{equation}
  \label{eq:3}
\sup_{0<t<t_1}t^{1/4}\|v(t)\|_{L^2_a}<\infty\quad\text{for a $t_1>0$.}
\end{equation}
For $t\ge t_1$, Proposition~\ref{prop:linear} and \eqref{eq:3} imply
\begin{equation}
  \label{eq:4}
\|e^{-t\mL_c}Q_cf\|_{L^2_a} \le
\|e^{-(t-t_1)\mL_c}Q_c\|_{B(L^2_a)}\|v(t_1)\|_{L^2_a}
\lesssim e^{-b(t-t_1)}\|f\|_{L^1_a}\,.
\end{equation}
Combining \eqref{eq:3} and \eqref{eq:4},
we obtain \eqref{eq:smoothing-a} with $j=0$.
We can prove \eqref{eq:smoothing-a} with $j=1$ and \eqref{eq:smoothing-b}
in exactly the same way.
Thus we complete the proof.
\end{proof}
\begin{rem}
Suppose $g(t)\in C([0,T];L^1_a)$ and that $v(t)$ is a solution of
\begin{equation}
\label{eq:rem-1}
 \pd_tv+\mL_cv=Q_c\pd_yf\,,\quad v(0)=0\,,
\end{equation}
in the class $C([0,T];L^2_a)$.
Then $v(t)$ can be represented as
\begin{equation}
  \label{eq:rem-2}
v(t)=\int_0^t e^{-(t-s)\mL_c}Q_cg(s)\,ds
\quad\text{for $t\in[0,T]$.}
\end{equation}
Indeed, Corollary~\ref{cor:s-smoothing} ensures that
the right hand side of \eqref{eq:rem-2} is a solution of
\eqref{eq:rem-1} in the class $C([0,T];L^2_a)$.
Note that a solution to \eqref{eq:rem-1} is unique in the class
$C([0,T];L^2_a)$.
Applying Corollary~\ref{cor:s-smoothing} to \eqref{eq:rem-2}
we have
\begin{equation}
  \label{eq:rem-3}
\|v(t)\|_{L^2_a}\le K'\int_0^t (t-s)^{-3/4}e^{-b'(t-s)}\|g(s)\|_{L^1_a}
\,ds\,.  
\end{equation}
In Section~\ref{sec:v2}, we will use \eqref{eq:rem-3} to estimate
quadratic nonlinearities.
\end{rem}

To deal with cubic terms of mKdV, we will use the following local smoothing
effect of
$$Ag(t):=\int_0^t e^{-(t-s)\mL_c}Q_cg(s)\,ds$$
 in the exponentially weighted space.
\begin{prp}
\label{prop:l-smoothing}
Let $p=2$ or $3$, $c>0$  and $0<a<\sqrt{c}$.
Then there exists a positive constant $K_1$ such that for every
$g\in L^2(\R_+;L^2_a)$,
\begin{equation}
  \label{eq:l-smooth2}
\|Ag\|_{L^2(\R_+;H^2_a)} \le K_1\|g\|_{L^2(\R_+;L^2_a)}\,.
\end{equation}
\end{prp}
Because of parabolic nature of $e^{-t\pd_x^3}$ on exponential weighted spaces,
Proposition~\ref{prop:l-smoothing} follows from the argument of \cite{dSimon}.
However, we here follow the lines of the proof of 
\cite[Proposition~2.7]{Burq-Tz}.
\begin{proof}[Proof of Proposition~\ref{prop:l-smoothing}]
  Fix $T>0$ and define $g_{T}(t)$ to be equal to $g(t)$ for
  $t\in[0,T]$ and $0$ elsewhere.  Define
$$
u_{T}(t)=\int_0^t e^{-(t-s)\mL_c}Q_cg_T(s)\,ds.
$$
Then $u_{T}(t)=Ag(t)$ for $t\in[0,T]$, $u(t)=0$ for $t\leq 0$ and $u_{T}(t)$ is a constant for $t>T$.
We have for $t\in\R$,
$$(\partial_t+\mL_c)u_{T}(t)=Q_c g_T(t)\,.$$
Thanks to the properties of $g_T$ and $u_T$, we can take the Fourier
transform in time (in the lower complex half plane) to get the relation
$$ (i\tau+\mL_c)\widehat{u_{T}}(\tau)
=Q_c \widehat{g_T}(\tau),\quad {\rm Im}(\tau)<0.$$
Take $\tau=\lambda-i\varepsilon$, $\varepsilon>0$ and $\lambda\in\R$. 
Using the resolvent estimate \eqref{eq:r-bound}, letting $\varepsilon$ to zero and integrating over $\lambda$, we get
$$
\|\widehat{u_{T}}(\lambda)\|_{L^2(\R;H^2_a)} \lesssim\|
 \widehat{g_T}(\lambda)\|_{L^2(\R;L^2_a)}\,.
$$
Since the Fourier transform of a function from $\R$ to a Hilbert space $H$
is an isometry on $L^2(\R;H)$, we obtain
$$\|u_{T}\|_{L^2(\R;H^2_a)} \lesssim\|g_T\|_{L^2(\R;L^2_a)}\,.$$
This is turn implies 
$$
\|Ag\|_{L^2([0,T];H^2_a)} \lesssim \|g\|_{L^2(\R;L^2_a)}\,.
$$
Observe that the implicit constant in the last inequality
is independent of $T$. Letting $T$ tends to infinity, we complete the proof of
Proposition~\ref{prop:l-smoothing}.
\end{proof}

\section{Decomposition of solutions around $1$-solitons}
\label{sec:decomp}
In this section, we will decompose a solution around $1$-solitons
into a sum of a modulating solitary wave and the remainder part. Let 
\begin{equation}
  \label{eq:decomp1}
u(t,x)=\varphi_{c(t)}(y)+v(t,y)\,,\quad y=x-x(t)\,.
\end{equation}
Here $c(t)$ and $x(t)$ denote the modulating speed and the modulating phase 
shift of the main solitary wave at time $t$ and $v(t,y)$ denotes
the remainder part of the solution.
Substituting \eqref{eq:decomp1} into \eqref{eq:gKdV},
we obtain
\begin{equation}
  \label{eq:v}
\pd_tv+\mL_{c(t)}v-(\dot{x}(t)-c(t))\pd_yv+\ell(t)+\pd_y\mathcal{N}=0\,,
\end{equation}
where
\begin{gather*}
\mathcal{N}=f(\varphi_{c(t)}+v)-f(\varphi_{c(t)})-f'(\varphi_{c(t)})v\,,\\
\ell(t)=\dot{c}(t)\pd_c\varphi_{c(t)}(y)
-(\dot{x}(t)-c(t))\pd_y\varphi_{c(t)}(y)\,.
\end{gather*}
Suppose that $u(t,x)$ satisfies the initial condition
$$u(0,x)=\varphi_{c_0}(x)+v_0(x)\,.$$
To apply the semigroup estimate directly to $v$ as \cite{PW},
the perturbation $v(t)$ should belong to an exponentially weighted space.
In order to extend Pego-Weinstein's approach for  $v_0\in L^2(\R)$ 
or $v_0\in H^1(\R)$, 
we further decompose $v(t,y)$ into a sum of a small $L^2$-solution of
the KdV equation and an exponentially localized part.
More precisely, let $\tv_1$ be a solution of \eqref{eq:gKdV} satisfying
$\tv_1(0,\cdot)=v_0$ and let
$$v_1(t,y)=\tv_1(t,x)\,,\quad v(t,y)=v_1(t,y)+v_2(t,y)\,.$$
Then
\begin{equation}
  \label{eq:v1}
  \left\{
    \begin{aligned}
& \pd_tv_1-\dot{x}(t)\pd_yv_1+\pd_y^3v_1+\pd_yf(v_1)=0\,,\\
& v_1(0,y)=v_0(y+x(0))\,,
    \end{aligned}\right.
\end{equation}
and
\begin{equation}
  \label{eq:v2}\left\{
  \begin{aligned}
& \pd_tv_2+\mL_{c(t)}v_2-\left(\dot{x}(t)-c(t)\right)\pd_yv_2+\ell(t)+\pd_yN(t)=0\,,\\
& v_2(0,x)=0\,,
  \end{aligned}\right.
\end{equation}
where $N(t)=N_1(t)+N_2(t)$, 
\begin{gather*}
N_1(t)=f(\varphi_{c(t)}+v_1)-f(\varphi_{c(t)})-f(v_1)\,,\\
N_2(t)=f(\varphi_{c(t)}+v)-f(\varphi_{c(t)}+v_1)-f'(\varphi_{c(t)})v_2\,.
\end{gather*}
The solutions of \eqref{eq:v1} will be evaluated by using a virial
estimate first used in the fundamental article by Kato (\cite{K}). The
solutions of \eqref{eq:v2} will be estimated by using the linear estimates
due to Pego-Weinstein in Section~\ref{sec:linear}.
\par
To begin with, we will show that $v_2(t)$ remains in exponentially
weighted spaces as long as the decomposition \eqref{eq:decomp1} exists
and $c(t)-c_0$ remains small.
\begin{lmm}
  \label{lem:difu-v1}
Let $p=2$ and $v_0\in L^2(\R)$ or $p=3$ and $v_0\in H^1(\R)$.
Suppose that $u(t)$ is a solution to \eqref{eq:gKdV} satisfying
$u(0)=\varphi_{c_0}+v_0$ and that and $\tv_1(t)$ is a solution to \eqref{eq:gKdV}
satisfying $\tv_1(0)=v_0$.
Then for $u(t)-\tv_1(t)\in C([0,\infty);L^2_a)$ for any $a\in [0,\sqrt{c_0})$.
\end{lmm}
\begin{proof}
Suppose $p=2$ and $\bar{u}(t,x)=u(t,x)-\tv_1(t,x)$. Then
\begin{equation}
  \label{eq:baru}
\left\{\begin{aligned}
&\pd_t\bar{u}+\pd_x^3\bar{u}+3\pd_x\left\{(u+\tv_1)\bar{u}\right\}=0
\quad\text{for $t>0$ and $x\in\R$,}\\
& \bar{u}(0,x)=\varphi_{c_0}(x)\quad \text{for $x\in\R$.}
  \end{aligned}\right.
\end{equation}
Thanks to the well-posedness of the KdV equation in $L^2(\R)$
, we have $\bar{u}(t)=u(t)-\tv_1(t)\in  C(\R;L^2(\R))$.
\par
Next we will show that $\bar{u}(t)\in L^\infty(0,T;L^2_a)$
for any $a\in(0,\sqrt{c_0})$ and $T>0$.
Suppose in addition that $v_0\in H^3(\R)\cap L^2_a$
so that $\bar{u}\in C(\R;H^3(\R))$ and $\tv_1(t)$, $u(t)\in
C([0,\infty);L^2_a)$.
At least formally, we have 
\begin{align*}
&\frac{d}{dt}\int_\R e^{2ax}\bar{u}^2(t,x)\,dx
+6a\int_\R e^{2ax}(\pd_x\bar{u})^2(t,x)\,dx
\\=&\int e^{2ax}\{(2a)^3\bar{u}^2+8a\bar{u}^3\}(t,x)\,dx+I\,,
\end{align*}
where
$I=12\int_\R e^{2ax}(2a\bar{u}^2+\bar{u}\pd_x\bar{u})\tv_1\,dx$.
Since $\|\tv_1(t)\|_{L^2}=\|v_0\|_{L^2}$ and
$$\|\bar{u}(t)\|_{L^2}\le \|u(t)\|_{L^2}+\|\tv_1(t)\|_{L^2}\le
\|\varphi_{c_0}\|_{L^2}+2\|v_0\|_{L^2}
\quad\text{for every $t\in\R$,}$$ 
there exists a $C>0$ depending only
on $c_0$ and $\|v_0\|_{L^2}$ such that 
$$\frac{d}{dt}\|\bar{u}(t)\|_{L^2_a}^2\le 
C\|\bar{u}(t)\|_{L^2_a}^2\,.$$
Here we use a weighted Sobolev inequality 
$\|w\|_{L^q_a}\lesssim \|w\|_{L^2_a}^{1/2+1/q}\|w\|_{H^1_a}^{1/2-1/q}$
for $2\le q \le \infty$ (see \eqref{eq:w2} in Section~\ref{sec:ap1}).
Thus we have
\begin{equation}
  \label{eq:baru-b}
\|\bar{u}(t)\|_{L^2_a}^2\le 
e^{Ct} \|\bar{u}(0)\|_{L^2_a}^2
=e^{Ct}\|\varphi_{c_0}\|_{L^2_a}^2\quad\text{for $t\ge0$.}  
\end{equation}
More precisely, let $\tilde{\chi}_n(x)=e^{2an}(1+\tanh a(x-n))/2$.
We have $\tilde{\chi}_n(x)\uparrow e^{2ax}$ and $\tilde{\chi}_n'(x)\uparrow 2ae^{2ax}$
as $n\to\infty$ and $0<\tilde{\chi}_n'(x)\le a\tilde{\chi}_n(x)$ and
$|\tilde{\chi}_n'''(x)|\le 4a^2\tilde{\chi}_n'(x)$ for any $x\in\R$.
Using the above properties of $\tilde{\chi}_n$ and Lemma~\ref{lem:weight},
we can easily justify \eqref{eq:baru-b} for $v_0\in H^3(\R)\cap L^2_a$. 
\par
For any $v_0\in L^2(\R)$, there exists a sequence $v_{0n}\in
H^3(\R)\cap L^2_a$ such that $v_{0n}\to v_0$ in $L^2(\R)$
as $n\to\infty$. Let $u_n(t)$ and $\tv_n(t)$ be a solution to \eqref{eq:gKdV}
satisfying $u_n(0)=\varphi_{c_0}+v_{0n}$ and $\tv_n(0)=v_{0n}$.
Then for any $t\in\R$, we have
$\lim_{n\to\infty}\|u_n(t)-\tv_n(t)-\bar{u}(t)\|_{L^2(\R)}=0$ and 
there exists a subsequence of $\{u_n(t)-\tv_n(t)\}$ that 
converges to $\bar{u}(t)$ weakly in $L^2_a$.
Thus we have
$$\|\bar{u}(t)\|_{L^2_a}
\le \liminf_{n\to\infty}\|u_n(t)-\tv_n(t)\|_{L^2_a}
\le e^{Ct/2}\|\varphi_{c_0}\|_{L^2_a}\,.$$
\par
By the variation of the constant formula, 
$$\bar{u}(t)=e^{-t\pd_x^3}\varphi_{c_0}
-3\int_0^te^{-(t-s)\pd_x^3}\pd_x(u(s)+\tv_1(s))\bar{u}(s)\,ds\,.$$
Since $e^{-t\pd_x^3}$ is a $C_0$-semigroup on $L^2_a$
and $\|\pd_xe^{-t\pd_x^3}\|_{B(L^1_a;L^2_a)}\lesssim t^{-3/4}$,
we easily see that $\bar{u}(t)\in C([0,\infty);L^2_a)$.
\par
The case $p=3$ can be shown in the same way.
Thus we complete the proof.
\end{proof}
Now we impose the symplectic orthogonality condition on $v_2$.
\begin{gather}
  \label{eq:orth1}
\int_\R v_2(t,y)\zeta_{c(t)}^1(y)\,dy=0\,,\\
  \label{eq:orth2}
\int_\R v_2(t,y)\zeta_{c(t)}^2(y)\,dy=0\,.
\end{gather}
Note that $\zeta_{c(t)}^1$, $\zeta_{c(t)}^2\in L^2_{-a}$ 
and $v_2(t)\in L^2_a$ for $a\in(0,\sqrt{c_0}/2)$ by Lemma~\ref{lem:difu-v1}
as long as $|c(t)-c_0|$ remains small.
In an $L^2_a$-neighborhood of $\varphi_{c_0}$, the speed and the phase satisfying
the orthogonality conditions can be uniquely chosen.

\begin{lmm}
\label{lem:imp}
Let $c_0>0$, $a\in(0,\sqrt{c_0})$.
Then there exist positive constants
$\delta_0$ and $\delta_1$ such that for each 
$w\in U_0(\delta_0):=\{w\in L^2_a\mid\|w-\varphi_{c_0}\|_{L^2_a}<\delta_0\}$,
there exists a unique $(\gamma\,,\,c)\in U_1(\delta_1):=
\{(\gamma\,,\,c)\in\R^2 \mid |\gamma|+|c-c_0|<\delta_1\}$ such that
\begin{equation}
  \label{eq:imp}
\la w(\cdot+\gamma)-\varphi_c\,,\,\zeta_c^1\ra=
\la w(\cdot+\gamma)-\varphi_c\,,\,\zeta_c^2\ra=0\,.  
\end{equation}
\end{lmm}
\begin{proof}
Let $G: L^2_a\times \R\times\R_+\to \R^2$ be a mapping defined by
\begin{equation}
  \label{eq:map-G}
G(w,\gamma,c)=
\begin{pmatrix}
\la w-\varphi_c(\cdot-\gamma)\,,\,\zeta_c^1(\cdot-\gamma)\ra
\\ \la w-\varphi_c(\cdot-\gamma)\,,\,\zeta_c^2(\cdot-\gamma)\ra
\end{pmatrix}\,.  
\end{equation}
Since
$G(\varphi_{c_0},0,c_0)={}^t(0,0)$ and $\nabla_{(\gamma,c)}G(\varphi_{c_0},0,c_0)=
\begin{pmatrix}1 & 0 \\ 0 & -1\end{pmatrix}$ is invertible,
Lemma~\ref{lem:imp} follows immediately from the implicit function theorem.
\end{proof}
Lemma~\ref{lem:imp} implies that the decomposition
\begin{equation}
  \label{eq:decomp2}
u(t,x)-\tv_1(t,x)=\varphi_{c(t)}(y)+v_2(t,y)\,,\quad y=x-x(t)
\end{equation}
satisfying the orthogonality conditions \eqref{eq:orth1} and \eqref{eq:orth2}
persists as long as $\varphi_{c(t)}(y)+v_2(t,y)$ stays in
$U_0(\delta_0)$ and $c(t)-c_0$ remains small.
\begin{lmm}
\label{lem:decomp}
Let $c_0$, $\delta_0$ and $\delta_1$ be as in Lemma~\ref{lem:imp}.
Suppose that $u$ and $\tv_1$ be solutions of \eqref{eq:gKdV} satisfying
$u(0)=\varphi_{c_0}+v_0$ and $\tv_1(0)=v_0\in L^2(\R)$.
Then there exist $T>0$ and $c(t)$, $x(t)\in C([0,T])\cap C^1((0,T))$
such that
$$c(0)=c_0\,,\quad x(0)=0\,,\quad \sup_{t\in[0,T]}|c(t)-c_0|<\delta_1\,,$$
and that $v_2$ defined by \eqref{eq:decomp2} satisfies
the orthogonality conditions \eqref{eq:orth1} and \eqref{eq:orth2}
for $t\in[0,T]$. Moreover, if $T$ is finite and
$$\sup_{t\in[0,T)}\|\varphi_{c(t)}+v_2(t)-\varphi_{c_0}\|_{L^2_a}<\delta_0\,,$$
then $T$ is not maximal.
\end{lmm}
\begin{proof}
Let $X$ be a Banach space with the norm
$\|u\|_X=\|(c_0-\pd_x^2)^{-2}u\|_{L^2_a}$. 
Then $G$ defined by \eqref{eq:map-G} is a smooth mapping from
$X\times \R_+\times\R$ to $\R^2$. Thus by the implicit function
theorem there exist an $X$-neighborhood $\widetilde{U}_0$ of $0$ and an
$\R^2$-neighborhood $\widetilde{U}_1$ of $(0,c_0)$ such that
for any $w\in \widetilde{U}_0$, there exists a unique
$(\gamma,c)\in\widetilde{U}_1$ satisfying \eqref{eq:imp}.
Moreover, the mapping
$\widetilde{U}_0\ni w\mapsto \Phi(w)=(\gamma,c)\in \widetilde{U}_1$ is smooth.
Since 
$$\bar{u}(t)=u(t)-\tv_1(t)\in C([0,\infty);L^2_a) \cap C^1((0,\infty);X)$$
by Lemma~\ref{lem:difu-v1} and \eqref{eq:gKdV}, we have
$(x(t)\,,\,c(t))=\Phi(\bar{u}(t))$ is $C^1$ on $(0,T)$. 
The other part of the proof is exactly the same as the proof of
\cite[Proposition~9.3]{MPQ}.
\end{proof}

In order to prove stability of $1$-solitons,  we will estimate
the following quantities in the subsequent sections.
Let
$$\|w\|_W:=\left(\int_\R e^{-2a|x|}w^2(x)\,dx\right)^{1/2}\,,\quad
\|w\|_{W_1}:=\left(\|w\|_W^2+\|\pd_xw\|_W^2\right)^{1/2}\,,$$
and let
\begin{align*}
& \bM_1(T)= \sup_{t\in[0,T]}\|v_1(t)\|_{L^2}+\|v_1(t)\|_{L^2(0,T;W_1)}\,,\\
& \bM_2(T)= 
\begin{cases}
& \sup_{t\in[0,T]}\|v_2(t)\|_{L^2_a}+\|v_2\|_{L^2(0,T;L^2_a)} \text{ if $p=2$,}\\
& \sup_{t\in[0,T]}\|v_2(t)\|_{L^2_a}+\|v_2\|_{L^2(0,T;H^1_a)} \text{ if $p=3$,}
\end{cases}\\
& \bM_v(T)=\sup_{0\le t\le T}\|v(t)\|_{L^2}^2\,,\quad
\bM_c(T)=\sup_{t\in[0,T]}|c(t)-c_0|\,,\\
& \bM_x(T)=\sup_{t\in[0,T]}|\dot{x}(t)-c(t)|\,,\quad
\bM_\gamma(T)=\sup_{t\in[0,T]}|\dot{\gamma}(t)-c(t)|\,,
\\ & \bM_{tot}(T)=
\begin{cases}
& \bM_1(T)+\bM_2(T)+\bM_v(T)+\bM_c(T)+\bM_x(T)\text{ if $p=2$,}\\
& \bM_1(T)+\bM_2(T)+\bM_v(T)+\bM_c(T)+\bM_\gamma(T)\text{ if $p=3$,}  
\end{cases}
\end{align*}
where $\gamma(t)$ is a function which shall be introduced in
Lemma~\ref{lem:modulation-3}.
We remark that $\|v_1(t)\|_{W_1}$ measures the interaction between
the solitary wave and  $v_1$.

\section{Modulation equations on the speed and the phase shift}
\label{sec:modulation}
In this section, we will derive modulation equations on the speed parameter 
$c(t)$ and the phase shift parameter $x(t)$.

Differentiating \eqref{eq:orth1} and \eqref{eq:orth2} with respect to $t$ 
and substituting \eqref{eq:v2} into the resulting equation,
we have for $i=1$ and $2$,
\begin{multline*}
0= \frac{d}{dt}\la v_2(t)\,,\, \zeta_{c(t)}^i\ra
 \\= -\la v_2(t)\,,\,\mL_{c(t)}^*\zeta_{c(t)}^i\ra
-\la \ell,\zeta^i_{c(t)}\ra
+\dot{c}\la v_2\,,\,\pd_c\zeta^i_{c(t)}\ra
-(\dot{x}-c)\la v_2\,,\, \pd_y\zeta^i_{c(t)}\ra
+\la N\,,\,\pd_y\zeta^i_{c(t)}\ra\,.
\end{multline*}
By \eqref{eq:etas} and \eqref{eq:orth2}, we have
$\la v_2(t)\,,\, \mL_{c(t)}^*\zeta_{c(t)}^i\ra=0$ for $i=1$, $2$.
Hence it follows that
\begin{equation}
  \label{eq:modulation}
\mathcal{A}(t)  \begin{pmatrix}c(t)-\dot{x}(t) \\   \dot{c}(t)  \end{pmatrix}
=\begin{pmatrix}
\la N\,,\, \pd_y\zeta_{c(t)}^1\ra \\ \la N\,,\, \pd_y\zeta_{c(t)}^2\ra  
\end{pmatrix}\,,
\end{equation}
where 
$$\mathcal{A}(t)=I-
\begin{pmatrix}
\la v_2\,,\, \pd_y\zeta_c^1\ra & \enskip\la v_2\,,\, \pd_c\zeta_c^1\ra
\\ \la v_2\,,\, \pd_y\zeta_c^2\ra & \enskip\la v_2\,,\, \pd_c\zeta_c^2\ra
\end{pmatrix}\,.$$
The following lemma provides estimates for $c(t)$ and $x(t)$ in terms
of the weighted $L^2$-norms of $v_1$ and $v_2$.
\begin{lmm}
\label{lem:modulation}
Let $p=2$ and $c_0>0$ and $0<a<\sqrt{c_0}/2$.
Then there exists a positive constant
$\delta_2$ such that if the decomposition \eqref{eq:decomp2} satisfying
\eqref{eq:orth1} and \eqref{eq:orth2} exists on $[0,T]$
and $\bM_1(T)+\bM_2(T)+\bM_c(T)<\delta_2$, then for $t\in[0,T]$,
\begin{equation}
  \label{eq:x,c-est}
  \begin{split}
& |\dot{c}(t)|+|\dot{x}(t)-c(t)| \\ & \lesssim  \|v_1(t)\|_W
+\|v_2\|_{L^2_a}(\|v_1(t)\|_W+\|v_2(t)\|_{L^2_a})\,.    
  \end{split}
\end{equation}
Furthermore,
\begin{gather}
  \label{eq:c-refine}
\frac{d}{dt}\left\{c(t)
+\theta_1(c(t))\left\la v_1(t),\varphi_{c(t)}\right\ra\right\}
=O\left(\|v_1(t)\|_W^2+\|v_2(t)\|_{L^2_a}^2\right)\,,\\
\label{eq:Mcx}
\bM_c(T)+\bM_x(T)\lesssim \bM_1(T)+\bM_2(T)^2\,.
\end{gather}
\end{lmm}
\begin{lmm}
\label{lem:modulation-3}
Let $p=3$ and $c_0>0$ and $0<a<\sqrt{c_0}/2$. Suppose $v_0\in
H^1(\R)$.  Then there exists a positive constant $\delta_2$ such that
if the decomposition \eqref{eq:decomp2} satisfying \eqref{eq:orth1}
and \eqref{eq:orth2} exists on $[0,T]$ and
$\bM_1(T)+\bM_2(T)+\bM_v(T)+\bM_c(T)<\delta_2$, then for $t\in[0,T]$,
\begin{equation}
  \label{eq:x,c-est-3}
  \begin{split}
& |\dot{c}(t)|+|\dot{x}(t)-c(t)| \lesssim 
\|v_1(t)\|_W+\left(\|v_1(t)\|_W^2+\|v_2\|_{L^2_a}^2\right)
(1+\|v_1(t)\|_{L^\infty}+\|v(t)\|_{L^\infty})\,.
  \end{split}
\end{equation}
Furthermore,
\begin{equation}
  \label{eq:c-refine-3}
\frac{d}{dt}\left\{c(t)+
\theta_1(c(t))\left\la v_1(t),\varphi_{c(t)}\right\ra\right\}
= O\left(\|v_1(t)\|_{W_1}+\|v_2(t)\|_{H^1_a}\right)^2\,,
\end{equation}
and there exists a $C^1$-function $\gamma(t)$ such that $\gamma(0)=0$,
\begin{gather}
  \label{eq:dx-dg}
\dot{\gamma}(t)-\dot{x}(t)=O\left(
(\|v_1(t)\|_{W_1}+\|v_2(t)\|_{H^1_a})^2(\|v_1(t)\|_{L^2}+\|v(t)\|_{L^2})\right)\,,
\\ \label{eq:dg-c}
\dot{\gamma}(t)-c(t)=O\left(\|v_1\|_W+\|v_2(t)\|_{L^2_a}^2\right)\,,\\
  \label{eq:Mcg}
\bM_c(T)+\bM_\gamma(T)\lesssim \bM_1(T)+\bM_2(T)^2\,.
\end{gather}
\end{lmm}
As we will see in Section~\ref{sec:thms},
Lemmas~\ref{lem:modulation} and \ref{lem:modulation-3} imply
that the modulating speed $c(t)$ converges to a fixed speed as $t\to\infty$.
We will use \eqref{eq:c-refine} and \eqref{eq:c-refine-3} to show that $c(t)$ 
tends to a fixed speed speed $c_+$ as $t\to\infty$.
\begin{proof}[Proof of Lemma~\ref{lem:modulation}]
If $\delta_2$ is small enough, then $2a<\sqrt{c(t)}$ for all $t\in[0,T]$,
and it follows from \eqref{eq:defzetas} that for $i=1$ and $2$,
$\|\pd_c\zeta_{c(t)}^i\|_{L^2_{-a}}$ and
$\sup_{y\in\R}e^{2a|y|}\pd_y\zeta_{c(t)}^i(y)$ are uniformly bounded on $[0,T]$.
Thus we have for $i=1$ and $2$,
\begin{equation*}
|\la v_2,\pd_y\zeta_{c(t)}^i\ra|+|\la v_2,\pd_c\zeta_{c(t)}^i\ra|
\lesssim \|v_2\|_{L^2_a}\lesssim \delta_2\,.  
\end{equation*}
Moreover,
\begin{gather*}
|\la N_1,\pd_y\zeta_{c(t)}^i\ra|\lesssim \|v_1(t)\|_W\,,\quad
|\la N_2,\pd_y\zeta_{c(t)}^i\ra|\lesssim \|v_1(t)\|_W\|v_2\|_{L^2_a}\,,
+\|v_2\|_{L^2_a}^2\,.
\end{gather*}
because $N_1=6\varphi_cv_1$ and $N_2=6v_1v_2+3v_2^2$. 
Combining the above with \eqref{eq:modulation}, we obtain \eqref{eq:x,c-est}.
Moreover,
\begin{equation}
  \label{eq:dotc}
\begin{split}
\dot{c}(t)=& \la N,\pd_y\zeta_{c(t)}^2\ra\left(1+O(\|v_2(t)\|_{L^2_a})\right)
+O\left(\|v_2(t)\|_{L^2_a}|\la N,\pd_y\zeta_{c(t)}^1\ra|\right)
\\=& \la N_1,\pd_y\zeta_{c(t)}^2\ra+O(\|v_1(t)\|_W^2+\|v_2(t)\|_{L^2_a}^2)\,.
\end{split}  
\end{equation}
Next we will rewrite $\la N_1,\pd_y\zeta_{c}^2\ra$ as a sum of time derivative
$\theta_1(c(t))\la v_1(t,\cdot), \varphi_{c(t)}\ra$ and a remainder part
which is integrable in time.
Substituting \eqref{eq:v1} and integrating the resulting equation by parts,
we have
\begin{equation}
  \label{eq:v1normal}
\begin{split}
& \frac{d}{dt}\la v_1\,,\,\varphi_{c}\ra
-\dot{c}\la v_1\,,\, \pd_c\varphi_{c}\ra
+(\dot{x}-c)\la v_1\,,\, \varphi_{c}'\ra
\\ = & 
\la v_1\,,\, \varphi_c'''-c\varphi_c'\ra
+3\la v_1^2\,,\, \varphi_{c}'\ra\,.
\end{split}  
\end{equation}
By \eqref{eq:B} and \eqref{eq:v1normal},
\begin{align*}
\la N_1,\pd_y\zeta_{c}^2\ra=&
3\theta_1(c)\left\la v_1\,,\, \pd_y\left(\varphi_{c}^2\right)\right\ra
\\=& \theta_1(c)\la v_1\,,\, c\varphi_c'-\varphi_c'''\ra
\\=& -\frac{d}{dt}\left(\theta_1(c)\la v_1,\varphi_{c}\ra\right)
+\frac{d\theta_1(c)}{dt}\la v_1,\varphi_{c}\ra
\\ & + \theta_1(c)\{\dot{c}\la v_1\,,\, \pd_c\varphi_{c}\ra
-(\dot{x}-c)\la v_1\,,\, \varphi_{c}'\ra
+3\la v_1^2\,,\, \pd_y\varphi_{c}\ra\}\,.
\end{align*}
Substituting \eqref{eq:x,c-est} into the above, we have
$$\la N_1,\pd_y\zeta_{c}^2\ra
+\frac{d}{dt}\left(\theta_1(c)\la v_1,\varphi_{c}\ra\right)
=O(\|v_1\|_W^2+\|v_2\|_{L^2_a}^2)\,.$$
Thus \eqref{eq:c-refine} follows from \eqref{eq:dotc} and the above.
Eq.~\eqref{eq:Mcx} follows immediately from \eqref{eq:x,c-est} and
\eqref{eq:c-refine}.
\end{proof}
\begin{proof}[Proof of Lemma~\ref{lem:modulation-3}]
By the definition, we have
\begin{gather}
\label{eq:defN1} N_1=N_{11}+N_{12}\,,\quad
N_{11}=9\varphi_c^2v_1\,,\quad N_{12}=9\varphi_cv_1^2\,,\\
\label{eq:defN2} N_2=N_{21}+N_{22}\,,\quad
N_{21}=9\varphi_cv_2(2v_1+v_2)\,,\quad N_{22}=3v_2(3v_1^2+3v_1v_2+v_2^2)\,,
\end{gather}
and for $i=1$, $2$,
\begin{gather}
\label{eq:N1s}
|\la N_{11},\pd_y\zeta_{c(t)}^i\ra|\lesssim \|v_1(t)\|_W\,,\quad 
|\la N_{12},\pd_y\zeta_{c(t)}^i\ra|\lesssim \|v_1(t)\|_W^2\,,
\\ \label{eq:N21}
|\la N_{21},\pd_y\zeta_{c(t)}^i\ra|\lesssim
(\|v_1(t)\|_W+\|v_2(t)\|_{L^2_a})^2\,,
\\ \notag
|\la N_{22},\pd_y\zeta_{c(t)}^i\ra|\lesssim
(\|v_1(t)\|_W+\|v_2(t)\|_{L^2_a})^2\|v_2(t)\|_{L^\infty}\,,
\end{gather}
in the same way as in the proof of Lemma~\ref{lem:modulation}.
Combining the above with \eqref{eq:modulation}, we have \eqref{eq:x,c-est-3}.
\par 
Next we will show \eqref{eq:c-refine-3}.
Since $|\pd_y \zeta_{c(t)}^i(y)|\lesssim e^{-2a|y|}$,
Lemma~\ref{cl:winfty} implies that
\begin{equation}
  \label{eq:wsa}
\int_\R |\pd_y \zeta_{c(t)}^i(y)||v^3(t,y)|\,dy
\lesssim \|v(t)\|_{L^2} \|v(t)\|_{W_1}^2
\quad\text{for $i=1$, $2$.}
\end{equation}
By \eqref{eq:wsa} and the H\"older inequality,
\begin{equation}
  \label{eq:N22}
|\la N_{22},\pd_y\zeta_{c(t)}^i\ra|\lesssim
(\|v_1(t)\|_{L^2}+\|v(t)\|_{L^2})(\|v_1(t)\|_{W_1}+\|v_2(t)\|_{H^1_a})^2\,,
\end{equation}
whence
\begin{equation}
  \label{eq:x,c-est2}
  \begin{split}
&|\dot{c}(t)|+|\dot{x}(t)-c(t)|\\ \lesssim & \|v_1(t)\|_W
+(1+\|v_1(t)\|_{L^2}+\|v(t)\|_{L^2})(\|v_1(t)\|_{W_1}+\|v_2(t)\|_{H^1_a})^2      
  \end{split}
\end{equation}
follows from \eqref{eq:modulation}, \eqref{eq:N1s}, \eqref{eq:N21} and
\eqref{eq:N22}.

As in the proof of Lemma~\ref{lem:modulation}, we have
\begin{equation}
  \label{eq:N11}
  \begin{split}
&\la N_{11},\pd_y\zeta_{c(t)}^2\ra 
+\frac{d}{dt}\left(\theta_1(c)\la v_1,\varphi_c\ra\right)
\\ = & 
\la v_1,\varphi_c\ra\frac{d}{dt}\theta_1(c)
+\theta_1(c)\left\{\dot{c}\la v_1,\pd_c\varphi_c\ra
-(\dot{x}-c)\la v_1,\varphi_c'\ra +3\la v_1^3,\varphi_c'\ra\right\}
\\ =& O\left((\|v_1(t)\|_{W_1}+\|v_2(t)\|_{H^1_a})^2\right)\,.
  \end{split}
\end{equation}
In the last line, we use \eqref{eq:x,c-est2} and the fact that
$$|\la v_1^3,\varphi_{c(t)}'\ra|\lesssim \|v_1(t)\|_{L^2}\|v_1(t)\|_{W_1}^2
=\|v_0\|_{L^2}\|v_1(t)\|_{W_1}^2\,.$$
Combining \eqref{eq:modulation}, \eqref{eq:N1s}, \eqref{eq:N21}, \eqref{eq:N22} and \eqref{eq:N11},
we obtain \eqref{eq:c-refine-3}.
\par
Finally, we will show \eqref{eq:dx-dg} and \eqref{eq:dg-c}.
Let $\gamma(t)$ be a $C^1$-function satisfying
$$\gamma(0)=0\,,\quad c(t)-\dot{\gamma}(t)=(1,0)\mathcal{A}(t)^{-1}
\begin{pmatrix} \la N_1+N_{21},\pd_y\zeta_{c(t)}^1\ra
\\  \la N_1+N_{21},\pd_y\zeta_{c(t)}^2\ra\end{pmatrix}\,.$$
The \eqref{eq:dg-c} follows from \eqref{eq:N1s} and \eqref{eq:N21}.
In view of the definition of $\gamma$ and \eqref{eq:modulation},
\begin{align*}
  |\dot{\gamma}(t)-\dot{x}(t)| \lesssim &
  (1+\|v\|_{L^2_a})\sum_{i=1,2}|\la N_{22},\pd_y\zeta_{c(t)}^i\ra| \\
  \lesssim &
  (1+\|v\|_{L^2_a})(\|v_1\|_{L^2}+\|v_2\|_{L^2})(\|v_1\|_W+\|v_2\|_{H^1_a})^2\,.
\end{align*}
Eq.~\eqref{eq:Mcg} immediately follows from \eqref{eq:c-refine-3},
\eqref{eq:dx-dg} and \eqref{eq:dg-c}.
Thus we complete the proof.
\end{proof}

\section{The $L^2$-estimate of $v$}
\label{sec:L2}
In this section, we will estimate $v$  by using
the $L^2$-conservation law of the gKdV equation.
\begin{lmm}
  \label{lem:v-L2}
Suppose $v_0\in L^2(\R)$ if $p=2$ and $v_0\in H^1(\R)$ if $p=3$.
Then there exist positive constants $\delta_3$ and $C$ such that 
if \eqref{eq:decomp2} satisfying \eqref{eq:orth2} exists and
$\|v_0\|_{L^2}+\bM_2(T)+\bM_c(T)<\delta_3$ for a $T\in(0,\infty]$, then
$$\bM_v(T)\le C(\bM_c(T)+\|v_0\|_{L^2})\,.$$
\end{lmm}
\begin{proof}
Since $v_1(t,x-x(t))$ is a solution of \eqref{eq:gKdV} satisfying
$v_1(0,x)=v_0(x)$,
\begin{equation}
  \label{eq:L2v1}
  \|v_1(t)\|_{L^2}=\|v_0\|_{L^2}\,,
\end{equation}
as long as the decomposition \eqref{eq:decomp2} exists.

Let $u(t)$ be a solution of \eqref{eq:gKdV} satisfying $u(0)=\varphi_{c_0}+v_0$.
By the $L^2$-conservation law,
\begin{equation}
  \label{eq:L2conserve}
\|u(t)\|_{L^2}^2=\|\varphi_{c_0}+v_0\|_{L^2}^2
=\|\varphi_{c_0}\|_{L^2}^2+O(\|v_0\|_{L^2})\,.  
\end{equation}
Substituting \eqref{eq:decomp1} into the left hand side,  we have
\begin{equation}
\label{eq:L2expand}
\|u(t)\|_{L^2}^2=
\|\varphi_{c(t)}\|_{L^2}^2+2\int_\R\varphi_{c(t)}(y)v(t,y)\,dy+\|v(t)\|_{L^2}^2\,.
\end{equation}
By the orthogonality condition \eqref{eq:orth2}, 
\begin{equation}
  \label{eq:L2conserve1}
\int_\R\varphi_{c(t)}(y)v(t,y)\,dy=\int_\R\varphi_{c(t)}(y)v_1(t,y)\,dy\,.  
\end{equation}
Combining \eqref{eq:L2v1}--\eqref{eq:L2conserve1}, we obtain
\begin{align*}
\|v(t)\|_{L^2}^2\le & \left|\|\varphi_{c(t)}\|_{L^2}^2-\|\varphi_{c_0}\|_{L^2(\R)}^2
\right|+O(\|v_0\|_{L^2(\R)})  
\\ \lesssim & |c(t)-c_0|+\|v_0\|_{L^2}\,.
\end{align*}
Thus we complete the proof.
\end{proof}

\section{The virial estimate of $v_1$}
\label{sec:virial}
In this section, we will show that $\|v_1(t)\|_{W_1}$
is square integrable in time by using the virial identity.
\begin{lmm}
  \label{lem:v1-a}
Suppose $p=2$ and $v_0\in L^2(\R)$.
There exist positive constants $C$ and $\delta_4$ such that
if $\bM_2(T)+\bM_c(T)+\bM_x(T)+\|v_0\|_{L^2}<\delta_4$,
then $\bM_1(T)\le C\|v_0\|_{L^2}$.
\end{lmm}
\begin{lmm}
  \label{lem:v1-b}
Suppose $p=3$ and $v_0\in H^1(\R)$.
There exist positive constants $C$ and $\delta_4$ such that
if $\bM_1(T)+\bM_2(T)+\bM_c(T)+\bM_\gamma(T)+\|v_0\|_{L^2}<\delta_4$,
then $\bM_1(T)\le C\|v_0\|_{L^2}$.
\end{lmm}
Let us recall the virial identity for the KdV equation
which ensures that $v_1(t)\in L^2(\R_+;W_1)$.
Let $\chi_\eps(x)=1+\tanh\eps x$, $\tilde{x}(t)$ be a $C^1$ function and 
$$I_{x_0}(t)=\int_\R\chi_\eps(x-\tilde{x}(t)-x_0)\tv_1(t,x)^2\,dx\,.$$
Then we have the following.
\begin{lmm}
  \label{lem:virial}
Suppose $v_0\in L^2(\R)$ if $p=2$ and $v_0\in H^1(\R)$ if $p=3$.
For any $c_1>0$, there exist positive constants $\eps_0$ and $\delta$
such that if $\inf_t\tilde{x}'(t)\ge c_1$, $\eps\in(0,\eps_0)$ and 
$\|v_0\|_{L^2}<\delta$, then for any $x_0\in\R$,
$$I_{x_0}(t)+\nu
\int_0^t\int_\R\chi_\eps'(x-\tilde{x}(s)-x_0)\{(\pd_x\tv_1)^2+\tv_1^2\}(s,x)\,dxds
\le I_{x_0}(0)\,,$$
where $\nu=\frac{1}{2}\min\{3,c_1\}$.
\end{lmm}
\begin{proof}[Proof of Lemmas~\ref{lem:v1-a} and \ref{lem:v1-b}]
Lemma~\ref{lem:v1-a} is an immediate consequence of 
the $L^2$-conservation law \eqref{eq:L2v1}
Lemma~\ref{lem:virial} with $\tilde{x}(t)=x(t)$ and $x_0=0$.
\par
To prove Lemma~\ref{lem:v1-b}, we apply Lemma~\ref{lem:virial}
with $\tilde{x}(t)=\gamma(t)$ and $x_0=0$.
Then 
$$\int_0^t\int_\R\chi_\eps'(y+h(s))v_1(t,y)\,dyds\lesssim \|v_0\|_{L^2}^2\,,$$
where $h(t)=x(t)-\gamma(t)$.
By Lemma~\ref{lem:modulation-3},
\begin{equation}
  \label{eq:h(t)}
|h(t)|\le \int_0^t|\dot{x}(t)-\dot{\gamma}(t)|\,dt\lesssim
\bM_1(T)^2+\bM_2(T)^2\,,  
\end{equation}
and there exists a positive constant $\mu$ depending only on $\delta_4$
such that $\chi_\eps'(y)\le \mu\chi_\eps'(y+h(t))$
for every $y\in\R$ and $t\in[0,T]$.
Thus we complete the proof.
\end{proof}

\begin{proof}[Proof of Lemma~\ref{lem:virial}]
Suppose that $\tv_1(t)$ is a smooth solution of \eqref{eq:gKdV}.
Then
\begin{equation}
  \label{eq:vir1}
  \begin{split}
& I_{x_0}'(t)
+\int_\R \chi_\eps'(x-\tilde{x}(t)-x_0)
\left\{3(\pd_x\tv_1)^2+\tilde{x}'(t)\tv_1^2
-\frac{6p}{p+1}\tv_1^{p+1}\right\}(t,x)\,dx
\\ =& \int_\R \chi_\eps'''(x-\tilde{x}(t)-x_0)\tv_1(t,x)^2\,dx\,.    
  \end{split}
\end{equation}
By the definition of $\chi_\eps$,
\begin{equation}
  \label{eq:ws0}
 0<\chi_\eps'(x)<2\eps\chi_\eps(x)\,,\enskip
|\chi_\eps''(x)|\le2\eps\chi_\eps'(x)\,,\enskip
|\chi_\eps'''(x)|\le4\eps^2\chi_\eps'(x)\quad\text{for $\forall x\in\R$.}
\end{equation}
Integrating \eqref{eq:vir1} over $[0,t]$ and using
Lemma~\ref{lem:weight}, \eqref{eq:L2v1} and \eqref{eq:ws0} to
the resulting equation, we obtain
\begin{equation}
  \label{eq:vir2}
  \begin{split}
& I_{x_0}(t)+\nu\int_0^t \int_\R \chi_\eps'(x-\tilde{x}(s)-x_0)
\left((\pd_x\tv_1)^2+\tv_1^2\right)(s,x)\,dxds \ \le I_{x_0}(0)
  \end{split}
\end{equation}
provided $\eps$ and $\delta_4$ are sufficiently small.
Since \eqref{eq:gKdV} is well-posed in $L^2(\R)$ if $p=2$ and in $H^1(\R)$
if $p=3$, we can verify \eqref{eq:vir2}
for any $v_0$ satisfying the assumption of Lemma~\ref{lem:virial}.
\end{proof}
\begin{crl}
  \label{cor:vir1}
Under the conditions of Lemma~\ref{lem:virial}, if there exists a positive
constant $\sigma$ such that $\inf_{t\ge0}\tilde{x}'(t)\ge c_1+\sigma$, then
$$\int_\R\chi_\eps(x-\tilde{x}(t))\tv_1(t,x)^2\,dx
\le \int_\R\chi_\eps(x-\tilde{x}(0)-\sigma t)v_0(x)^2\,dx\to0
\quad\text{as $t\to\infty$.}$$
\end{crl}
\begin{proof}
  Let $t_1>0$ and $\tilde{x}_1(t)=\tilde{x}(t)-\sigma(t-t_1)$. Then 
$\tilde{x}_1(t_1)=\tilde{x}(t_1)$ and $\tilde{x}_1'(t)\ge c_1$
for every $t\ge0$. Using $\tilde{x}_1(t)$ in place of $\tilde{x}(t)$
in Lemma~\ref{lem:virial}, we have
\begin{align*}
\int_\R\chi_\eps(x-\tilde{x}_1(t_1))\tv_1(t,x)^2\,dx
\le & \int_\R\chi_\eps(x-\tilde{x}_1(0))\tv_1(0,x)^2\,dx
\\ =& \int_\R\chi_\eps(x-\tilde{x}(0)-\sigma t_1)v_0(x)^2\,dx\,.
\end{align*}
Thus we complete the proof.
\end{proof}

\section{The weighted estimate of $v_2$}
\label{sec:v2}
In this section, we will estimate $\|v_2(t)\|_{L^2_a}$ by using the
exponential stability property of the linearized operator as in
\cite{M1,MPQ,PW}. Thanks to the parabolic smoothing effect of
$e^{t\pd_x^3}$ on $L^2_a$, we do not need re-centering argument as in
\cite{MPQ} which is used to avoid a derivative loss caused by the term
$(\dot{x}-c)\pd_yv$.

\begin{lmm}
  \label{lem:v2est-a} Let $p=2$.
There exist positive constants $C$ and  $\delta_5$ such that if
$\|v_0\|_{L^2}+\bM_{tot}(T)\le \delta_5$, then $\bM_{tot}(T)\le C\|v_0\|_{L^2}$.
\end{lmm}
\begin{proof}
To begin with, we deduce a priori bounds on $\bM_1$, $\bM_c$, $\bM_x$
and $\bM_v$ in terms of $\|v_0\|_{L^2}$ and $\bM_2(T)$.
Lemma~\ref{lem:v1-a} implies
\begin{equation}
  \label{eq:M1}
\bM_1(T)\lesssim \|v_0\|_{L^2}\,.  
\end{equation}
By \eqref{eq:Mcx} and \eqref{eq:M1},
\begin{equation}
  \label{eq:Mcx'}
\bM_c(T)+\bM_x(T)\lesssim  \|v_0\|_{L^2}+\bM_2(T)^2\,,
\end{equation}
and 
\begin{equation}
  \label{eq:Mv}
\bM_v(T)\lesssim \|v_0\|_{L^2}+\bM_c(T)\lesssim\|v_0\|_{L^2}+\bM_2(T)^2
\end{equation}
follows from Lemma~\ref{lem:v-L2} and \eqref{eq:Mcx'}.
Hence it suffices to show $\bM_2(T)\lesssim \|v_0\|_{L^2}$
to prove Lemma~\ref{lem:v2est-a}.
\par
Now we will estimate $v_2$.
Eq.~\eqref{eq:v2} can be rewritten as
\begin{equation}
  \label{eq:v2'}
  \pd_tv_2+\mL_{c_0}v_2+\ell(t)+\pd_y(N(t)+\widetilde{N}(t))=0\,,
\end{equation}
where $\widetilde{N}(t)=\left(c_0-\dot{x}(t)\right)v_2
+6\left(\varphi_{c(t)}-\varphi_{c_0}\right)v_2$.
Using the variation of constants formula, we have
\begin{equation}
  \label{eq:var-const}
Q_{c_0}v_2(t)=-\int_0^t e^{-(t-s)\mL_{c_0}}Q_{c_0}
\{\ell(s)+\pd_y(N(s)+\widetilde{N}(s))\}\,ds\,.  
\end{equation}
Applying Proposition~\ref{prop:linear} and Corollary~\ref{cor:s-smoothing}
to \eqref{eq:var-const}, we have
\begin{equation}
  \label{eq:v2-p1}
  \begin{split}
\|Q_{c_0}v_2(t)\|_{L^2_a}\lesssim & \int_0^t e^{-b(t-s)}\|\ell(s)\|_{L^2_a}\,ds
\\ & +
\int_0^te^{-b'(t-s)}(t-s)^{-1/2}(\|N_1(s)\|_{L^2_a}+\|\widetilde{N}(s)\|_{L^2_a})\,ds
\\ & +\int_0^te^{-b'(t-s)}(t-s)^{-3/4}\|N_2(s)\|_{L^1_a}\,ds\,.
  \end{split}
\end{equation}
Since $Q_{c(t)}v_2(t)=v_2(t)$ and $\|Q_{c(t)}-Q_{c_0}\|_{B(L^2_a)}=O(|c(t)-c_0|)$,
$$\|v_2(t)-Q_{c_0}v_2(t)\|_{L^2_a}=O(|c(t)-c_0|)\|v_2(t)\|_{L^2_a}\,.$$
Hence for small $\delta_5$,
there exist positive constants $d_1$ and $d_2$ such that
$$d_1\|v_2(t)\|_{L^2_a}\le \|Q_{c_0}v_2(t)\|_{L^2_a}\le d_2 \|v_2(t)\|_{L^2_a}
\quad\text{for $t\in[0,T]$.}$$
\par
Let us prove
\begin{gather}
  \label{eq:N1-w}
\|N_1\|_{L^\infty(0,T;L^2_a)}+\|N_1\|_{L^2(0,T;L^2_a)}\lesssim \|v_0\|_{L^2}\,,
\\   \label{eq:N2-w}
\|N_2\|_{L^\infty(0,T;L^1_a)}+\|N_2\|_{L^2(0,T;L^1_a)}\lesssim
(\|v_0\|_{L^2}^{1/2}+\bM_2(T))\bM_2(T)\,,
\\ \label{eq:wN}
\|\widetilde{N}\|_{L^\infty(0,T;L^2_a)}+\|\widetilde{N}\|_{L^2(0,T;L^2_a)}
\lesssim (\|v_0\|_{L^2}+\bM_2(T)^2)\bM_2(T)\,,
\\   \label{eq:l-w}
\|\ell\|_{L^\infty(0,T;L^2_a)}+\|\ell\|_{L^2(0,T;L^2_a)}\lesssim
\|v_0\|_{L^2}+\bM_2(T)^2\,.
\end{gather}
If $\delta_5$ is sufficiently small, we have
$2a<\inf_{s\in[0,T]}\sqrt{c(s)}$ and
\begin{equation}
  \label{eq:N1-a}
\|N_1(s)\|_{L^2_a}\lesssim \|v_1(s)\|_W  
\end{equation}
follows from the definition of $N_1$.
Since $|N_2|\lesssim |v_2|(|v_1|+|v_2|)$ and $v_2=v-v_1$,
we have for $s\in[0,T]$, 
\begin{equation}
  \label{eq:N2-a}
\begin{split}
\|N_2(s)\|_{L^1_a}\lesssim & \|v_2(s)\|_{L^2_a}(\|v_1(s)\|_{L^2}+\|v_2(s)\|_{L^2})
\\ \lesssim &  (\bM_1(T)+\bM_v(T)^{1/2})\|v_2(s)\|_{L^2_a}\,.
\end{split}  
\end{equation}
Combining \eqref{eq:M1}--\eqref{eq:Mv} with \eqref{eq:N1-a}
and \eqref{eq:N2-a},
we obtain \eqref{eq:N1-w} and \eqref{eq:N2-w}. Moreover,
\begin{align*}
\|\widetilde{N}(s)\|_{L^2_a}\lesssim & \left(
\left|\dot{x}(s)-c(s)\right|+\left|c(s)-c_0\right|\right)\|v_2(s)\|_{L^2_a}
\\ \lesssim  & \left(\|v_0\|_{L^2}+\bM_2(T)^2\right)\|v_2(s)\|_{L^2_a}\,.
\end{align*}
By \eqref{eq:x,c-est} and the definition of $\ell$,
\begin{align*}
\|\ell(s)\|_{L^2_a}\lesssim &
\left|\dot{c}(s)\right|+\left|\dot{x}(s)-c(s)\right|
\\ \lesssim  & \|v_1(s)\|_W+\|v_2(s)\|_{L^2_a}^2
\lesssim  \|v_1(s)\|_W+\bM_2(T)\|v_2(s)\|_{L^2_a}\,.  
\end{align*}
Thus we prove \eqref{eq:N1-w}--\eqref{eq:l-w}.
Since $e^{-bt}(1+t^{-3/4})\in L^1((0,\infty))$,
it follows from Young's inequality and \eqref{eq:v2-p1}--\eqref{eq:l-w} that
\begin{align*}
\bM_2(T)=& \|v_2\|_{L^\infty(0,T;L^2_a)}+\|v_2\|_{L^2(0,T;L^2_a)}
\\ \lesssim & \|v_0\|_{L^2}+(\|v_0\|_{L^2}^{1/2}+\bM_2(T))\bM_2(T)\,.
\end{align*}
Thus we have $\bM_2(T)\lesssim \|v_0\|_{L^2}$ provided $\delta_5$ is sufficiently
small. This completes the proof of Lemma~\ref{lem:v2est-a}.
\end{proof}

\begin{lmm}
  \label{lem:v2est-b}
Let $p=3$. There exists a positive constant $\delta_5$ such that if
$\bM_{tot}(T)+\linebreak\|v_0\|_{L^2}^{3/4}\|v_0\|_{H^1}^{1/4}<\delta_5$, then
$\bM_{tot}(T)\lesssim \|v_0\|_{L^2}$.
\end{lmm}
To prove Lemma~\ref{lem:v2est-b}, we need the $H^1$-bound of
$v_1$ and $v$.
\begin{lmm}
  \label{lem:v1-H1}
Let $p=3$ and $\tv_1$ be a solution of \eqref{eq:gKdV} satisfying
$\tv_1(0)=v_0\in H^1(\R)$. Then 
\begin{equation}
  \label{eq:H1-bound1}
\|\pd_x\tv_1(t)\|_{L^2}\le C\left(\|\pd_xv_0\|_{L^2}+\|v_0\|_{L^2}^3\right)\,,
\end{equation}
where $C$ is a constant independent of $t$ and $v_0$.
\end{lmm}
\begin{proof}
Since $\|\pd_xv_1\|_{L^2}^2\le 2E(v_1)+\frac{3}{2}\|v_1\|_{L^4}^4$
and $\|v_1\|_{L^4}\lesssim \|v_1\|_{L^2}^{3/4}\|\pd_xv_1\|_{L^2}^{1/4}$,
\begin{equation*}
\|\pd_xv_1(t)\|_{L^2}^2\le 2E(v_1(t))+\frac12\|\pd_xv_1(t)\|_{L^2}^2
+O(\|v(t)\|_{L^2}^6)\,.
\end{equation*}
Combining the above with the $L^2$ conservation law
and the energy conservation law, we obtain \eqref{eq:H1-bound1}.
\end{proof}
\begin{lmm}
  \label{lem:v-H1}
There exists a positive constant $\delta'$ such that if
$\|v_0\|_{L^2}+\bM_2(T)+\bM_v(T)+\bM_c(T)<\delta'$, then
\begin{equation}
  \label{eq:H1-bound2}
\|v(t)\|_{H^1}\le C(\|v_0\|_{H^1}+\|v(t)\|_{L^2}^3+|c(t)-c_0|)
\quad\text{for $t\in[0,T]$.}
\end{equation}
\end{lmm}
\begin{proof}
Let $S(u):=E(u)+\frac{c_0}{2}\|u\|_{L^2}^2$.
Thanks to the energy and the $L^2$ conservation laws,
\begin{equation}
  \label{eq:Sc0}
  \begin{split}
S(\varphi_{c_0}+v_0)=& S(\varphi_{c(t)}+v)
\\=& S(\varphi_{c(t)})+\la S'(\varphi_{c(t)}),v\ra
+\frac{1}{2}\la S''(\varphi_{c(t)})v,v\ra-R\,,
  \end{split}
\end{equation}
where
$$R=\frac34\int_\R\left(4\varphi_{c(t)}v^3+v^4\right)\,dy\,.$$
Since $S'(\varphi_{c_0})=0$ by \eqref{eq:B},
\begin{equation}
  \label{eq:Sc1}
S(\varphi_{c(t)})=S(\varphi_{c_0})+O(|c(t)-c_0|^2)\,.
\end{equation}
By \eqref{eq:orth2}, the fact that $S'(\varphi_{c(t)})=(c_0-c(t))\varphi_{c(t)}$
and \eqref{eq:L2v1},
\begin{equation}
  \label{eq:Sc2}
\la S'(\varphi_{c(t)}),v\ra=(c_0-c(t))\la v_1,\varphi_{c(t)}\ra
=O\left(|c(t)-c_0|\|v_0\|_{L^2}\right)\,.
\end{equation}
\par
Next, we will show that $S''(\varphi_c)$ is positive definite for $v_2$.
Let $L=S''(\varphi_c)+(c-c_0)I$ and
$$v_2=a\varphi_c^2+b\varphi_c'+p\,,\quad
\la p,\varphi_c^2\ra=\la p,\varphi_c'\ra=0\,.$$
Note that 
\begin{equation}
  \label{eq:S''}
L\varphi_c^2=-3c\varphi_c^2\,, \quad L\varphi_c'=0\,,  
\end{equation}
and that $L$ is positive definite on 
${}^\perp\spann\{\varphi_c^2\,,\varphi_c'\}$  by the Sturm-Liouville theorem.
By \eqref{eq:S''},
$$\la Lv_2,v_2\ra=\la Lp,p\ra
-3ca^2\la L\varphi_c^2,\varphi_c^2\ra\,.$$
Since $\la a\varphi_c^2+p,\varphi_c\ra=\la v_2,\varphi_c\ra=0$ by
\eqref{eq:orth2} and $d\|\varphi_c\|_{L^2}^2/dc>0$,
$$\la Lv_2,v_2\ra \gtrsim \|a\varphi_c^2+p\|_{H^1}^2$$
in exactly the same way as \cite[Proof of Theorem~3.3]{GSS}.
Thanks to the orthogonality condition \eqref{eq:orth1}, we have
$|b|\lesssim \|a\varphi_c^2+p\|_{H^1}$.
Thus there exists a positive constant $\nu$ such that
\begin{equation}
  \label{eq:S-v2}
\la S''(\varphi_c)v_2,v_2\ra\ge \nu\|v_2\|_{H^1}^2
\end{equation}
provided $|c-c_0|$ is sufficiently small.
\par
By \eqref{eq:S-v2} and Lemma~\ref{lem:v1-H1},
\begin{equation}
  \label{eq:Sc3}
\la S''(\varphi_{c(t)})v,v\ra\ge \frac{\nu}{2}\|v_2(t)\|_{H^1}^2
-O(\|v_1(t)\|_{H^1}^2)
\ge \frac{\nu}{2}\|v_2(t)\|_{H^1}^2-O(\|v_0\|_{H^1}^2)\,.
\end{equation}
By the Sobolev imbedding theorem,
\begin{equation}
  \label{eq:Sc4}
|R|\le \frac{\nu}{8}\|\pd_xv\|_{L^2}^2+O(\|v(t)\|_{L^2}^6)\,.
\end{equation}
Combining \eqref{eq:Sc0}--\eqref{eq:Sc2}, \eqref{eq:Sc3} and
\eqref{eq:Sc4}, we obtain \eqref{eq:H1-bound2}.
\end{proof}

Now we are in position to prove Lemma~\ref{lem:v2est-b}.
\begin{proof}[Proof of Lemma~\ref{lem:v2est-b}]
By Lemmas~\ref{lem:modulation-3}, \ref{lem:v-L2} and \ref{lem:v1-b},
\begin{gather}
  \label{eq:Ms}
  \bM_1(T)\lesssim \|v_0\|_{L^2}\,,\quad
\bM_c(T)+\bM_\gamma(T)+\bM_v(T)\lesssim \|v_0\|_{L^2}+\bM_2(T)^2\,.
\end{gather}
Furthermore, it follows from \eqref{eq:x,c-est2} and \eqref{eq:Ms} that
\begin{equation}
  \label{eq:dx-c3}
  \begin{split}
&\|\dot{c}\|_{L^1(0,T)+L^2(0,T)}+\|\dot{x}-c\|_{L^1(0,T)+L^2(0,T)}
\\ \lesssim & \bM_1(T)+\bM_2(T)^2\lesssim \|v_0\|_{L^2}+\bM_2(T)^2\,.
  \end{split}
\end{equation}
\par

Now we will estimate $\bM_2(T)$. Instead of $v_2$, we will estimate
its small translation.
Let $\tv_2(t,y)=v(t,y+h(t))$. By \eqref{eq:v2},
\begin{equation}
  \label{eq:tv2}\left\{
  \begin{aligned}
& \pd_t\tv_2+\mL_{c_0}\tv_2+\tau_{h(t)}\ell(t)
+\pd_y(\tau_{h(t)}N(t)+\widetilde{N}(t))=0\,,\\
& \tv_2(0,x)=0\,,
  \end{aligned}\right.  
\end{equation}
where $\tau_h$ is a shift operator defined by $\tau_hg(x)=g(x+h)$ and
$$\widetilde{N}(t)=\left(c_0-\dot{\gamma}(t)\right)\tv_2
+9\left(\tau_{h(t)}\varphi_{c(t)}^2-\varphi_{c_0}^2\right)\tv_2\,.$$
Using the variation of constants formula, we have
$Q_{c_0}\tv_2(t)=v_{21}(t)+v_{22}(t)+v_{23}(t)$, where
\begin{gather*}
v_{21}(t)=-\int_0^t e^{-(t-s)\mL_{c_0}}Q_{c_0}\tau_{h(s)}\ell(s)\,ds\,,\\
v_{22}(t)=-\int_0^t e^{-(t-s)\mL_{c_0}}Q_{c_0}\tau_{h(s)}
(N(s)+\widetilde{N}(s)-N_{22}(s))\,ds\,,\\
v_{23}(t)=-\int_0^t e^{-(t-s)\mL_{c_0}}Q_{c_0}\tau_{h(s)}N_{22}(s)\,ds\,.  
\end{gather*}
Note that $\|\tv_2(t)\|_{H^1_a}\lesssim \|Q_{c_0}\tv_2(t)\|_{H^1_a}$
as in the proof of Lemma~\ref{lem:v2est-a}.
By Proposition~\ref{prop:linear},
\begin{equation}
\label{eq:v21}
\|v_{21}(t)\|_{H^1_a} \lesssim 
\int_0^t e^{-b(t-s)}\|\tau_{h(s)}\ell(s)\|_{H^1_a}\,ds\,.
\end{equation}
By \eqref{eq:dx-c3}, \eqref{eq:h(t)} and the definition of $\ell$,
\begin{equation}
  \label{eq:l-3}
\|\tau_{h(t)}\ell(t)\|_{L^1(0,T;H^1_a)+L^2(0,T;H^1_a)}
\lesssim \|v_0\|_{L^2}+\bM_2(T)^2\,.
\end{equation}
Combining \eqref{eq:v21} and \eqref{eq:l-3}, we have
\begin{equation}
  \label{eq:v21b}
\|v_{21}\|_{L^\infty(0,T;H^1_a)}+\|v_{21}\|_{L^2(0,T;H^1_a)}
\lesssim \|v_0\|_{L^2}+\bM_2(T)^2\,.
\end{equation}
\par
Using Corollary~\ref{cor:s-smoothing},
we can estimate $\sup_{t\in[0,T]}\|v_{22}(t)\|_{L^2_a}$ in the same way
as the proof of Lemma~\ref{lem:v2est-a}. Indeed,
\begin{equation}
  \label{eq:tv2-p1}
  \begin{split}
& \|v_{22}(t)\|_{L^2_a}\lesssim 
\int_0^te^{-b'(t-s)}(t-s)^{-1/2}
\left(\|\tau_{h(s)}N_{11}(s)\|_{L^2_a}+\|\widetilde{N}(s)\|_{L^2_a}\right)\,ds
\\ & +\int_0^te^{-b'(t-s)}(t-s)^{-3/4}
\left(\|\tau_{h(s)}N_{12}(s)\|_{L^1_a}+\|\tau_{h(s)}N_{21}(s)\|_{L^1_a}\right)\,ds\,.
  \end{split}
\end{equation}
Now we will estimate each term of the right hand side of \eqref{eq:tv2-p1}.
Since 
$|\widetilde{N}(s)|\lesssim (|\dot{\gamma}(s)-c(s)|+|c(s)-c_0|)|v_2|$,
it follows from \eqref{eq:Ms}  and \eqref{eq:h(t)} that
\begin{equation}
  \label{eq:wN-3}
  \begin{split}
& \|\tau_{h(s)}\widetilde{N}\|_{L^\infty(0,T;L^2_a)} 
+\|\tau_{h(s)}\widetilde{N}\|_{L^2(0,T;L^2_a)}
\\ \lesssim &
(\bM_c(T)+\bM_\gamma(T))\bM_2(T) \lesssim (\|v_0\|_{L^2}+\bM_2(T)^2)\bM_2(T)\,.
  \end{split}
\end{equation}
By \eqref{eq:defN1}, \eqref{eq:defN2},  \eqref{eq:Ms} and \eqref{eq:h(t)},
\begin{gather}
  \label{eq:N11'}
\|\tau_{h(s)}N_{11}\|_{L^\infty(0,T;L^2_a)}+\|\tau_{h(s)}N_{11}\|_{L^2(0,T;L^2_a)}
\lesssim \bM_1(T) \lesssim \|v_0\|_{L^2}\,,
\\ \label{eq:N12'}
\|\tau_{h(s)}N_{12}\|_{L^\infty(0,T;L^1_a)}
\lesssim \bM_1(T)^2\lesssim \|v_0\|_{L^2}^2\,,
\end{gather}
\begin{equation}
  \label{eq:N21'}
\|\tau_{h(s)}N_{21}\|_{L^\infty(0,T;L^1_a)}\lesssim 
\bM_1(T)(\bM_1(T)+\bM_2(T))\lesssim \|v_0\|_{L^2}^2+\bM_2(T)^2\,.
\end{equation}
Combining \eqref{eq:tv2-p1}--\eqref{eq:N21'} with Young's inequality,
we have
\begin{equation}
  \label{eq:v22a}
\|v_{22}\|_{L^\infty(0,T;L^2_a)}\lesssim \|v_0\|_{L^2}+\bM_2(T)^2\,.
\end{equation}
\par
Next we will estimate $\|v_2\|_{L^2(0,T;H^1_a)}$.
By \eqref{eq:w2} in Lemma~\ref{cl:winfty},
\begin{equation}
  \label{eq:N12-21''}
\begin{split}
& \|\tau_{h(s)}N_{12}\|_{L^2(0,T;L^2_a)}+\|\tau_{h(s)}N_{21}\|_{L^2(0,T;L^2_a)}
\\ \lesssim &
\bM_1(T)(\bM_1(T)+\bM_2(T))\lesssim \|v_0\|_{L^2}^2+\bM_2(T)^2\,.
\end{split}
\end{equation}
Combining Corollary~\ref{cor:s-smoothing} with \eqref{eq:wN-3},
\eqref{eq:N11'} and \eqref{eq:N12-21''}, we have
\begin{equation}
  \label{eq:v22b}
\|v_{22}\|_{L^\infty(0,T;H^1_a)}  \lesssim \|v_0\|_{L^2}+\bM_2(T)^2\,.
\end{equation}
\par
Finally, we will estimate $v_{23}(t)$.
By Corollary~\ref{cor:s-smoothing} and \eqref{eq:h(t)},
\begin{equation}
\label{eq:v23a}
\begin{split}
\|v_{23}(t)\|_{L^2_a}\lesssim &
\int_0^te^{-b'(t-s)}(t-s)^{-1/2}\|\tau_{h(s)}N_{22}(s)\|_{L^2_a}\,ds
\\ \lesssim & \int_0^te^{-b'(t-s)}(t-s)^{-1/2}
\left(\|v_1^2v_2(s)\|_{L^2_a}+\|v_2^3(s)\|_{L^2_a}\right)\,ds\,.
\end{split}
\end{equation}
Since $\|f\|_{L^4}\lesssim \|f\|_{L^2}^{3/4}\|\pd_xf\|_{L^2}^{1/4}$
and $\|v_1(t)\|_{L^2}=\|v_0\|_{L^2}$ is small,
Lemma~\ref{lem:v1-H1} implies\linebreak
$\|v_1(t)\|_{L^4}\lesssim \|v_0\|_{L^2}^{3/4}\|v_0\|_{H^1}^{1/4}$.
Hence by \eqref{eq:w2},
\begin{equation*}
\|v_1^2v_2\|_{L^2_a}\lesssim \|v_1\|_{L^4}^2\|v_2\|_{L^\infty_a}
\lesssim
\|v_0\|_{L^2}^{3/2}\|v_0\|_{H^1}^{1/2}\|v_2\|_{L^2_a}^{1/2}\|v_2\|_{H^1_a}^{1/2}\,.
\end{equation*}
By the definition of $\bM_2(T)$ with $p=3$,
\begin{equation}
\label{eq:N22-1}
\|v_1^2v_2\|_{L^2(0,T;L^2_a)}+\|v_1^2v_2\|_{L^4(0,T;L^2_a)}
\lesssim \|v_0\|_{L^2}^{3/2}\|v_0\|_{H^1}^{1/2}\bM_2(T)\,.
\end{equation}
Lemmas~\ref{lem:v1-H1} and \ref{lem:v-H1}, \eqref{eq:L2v1}
and \eqref{eq:Ms} imply 
\begin{equation}
  \label{eq:v-H1'}
  \begin{split}
& \|v_2(t)\|_{L^2}\le \|v(t)\|_{L^2}+\|v_1(t)\|_{L^2}\lesssim
\|v_0\|_{L^2}^{1/2}+\bM(T)\,,\\
& \|v_2(t)\|_{H^1}\le \|v(t)\|_{H^1}+\|v_1(t)\|_{H^1}
\lesssim \|v_0\|_{H^1}+\bM_2(T)^2\,.    
  \end{split}
\end{equation}
Combining \eqref{eq:Ms}, \eqref{eq:v-H1'} and \eqref{eq:w1} 
in Section~\ref{sec:ap1} with $\theta=5/7$, we have
\begin{align*}
\|v_2^3\|_{L^2_a}\le &  \|v_2\|_{L^2}\|v_2^2\|_{L^\infty_a}
\\ \lesssim & \|v_2\|_{H^1_a}^{5/7}\|v_2\|_{L^2_a}^{2/7}
\|v_2\|_{L^2}^{12/7}\|v_2\|_{H^1}^{2/7}
\\ \lesssim & \|v_2\|_{H^1_a}^{5/7}\|v_2\|_{L^2_a}^{2/7}
(\|v_0\|_{L^2}^{6/7}+\bM_2(T)^{12/7})(\|v_0\|_{H^1}^{2/7}+\bM_2(T)^{4/7})\,.
\end{align*}  
Since $\left\|\|v_2\|_{H^1_a}^{5/7}\|v_2\|_{L^2_a}^{2/7}\right\|_{L^2(0,T)\cap L^{14/5}(0,T)}
\lesssim \bM_2(T)$, 
\begin{equation}
  \label{eq:N22-2}
\begin{split}
& \|v_2^3\|_{L^2(0,T;L^2_a)}+\|v_2^3\|_{L^{14/5}(0,T;L^2_a)}
\\ \lesssim &
(\|v_0\|_{L^2}^{6/7}+\bM_2(T)^{12/7})(\|v_0\|_{H^1}^{2/7}+\bM_2(T)^{4/7})\bM_2(T)
\\ \lesssim & \|v_0\|_{L^2}^{6/7}\|v_0\|_{H^1}^{2/7}
\left(1+\frac{\bM_2(T)^2}{\|v_0\|_{L^2}}\right)^{6/7}\bM_2(T)
\\ & +\|v_0\|_{L^2}^{6/7}\bM_2(T)^{11/7}+\bM_2(T)^{23/7}\,.
\end{split}
\end{equation}
Substituting \eqref{eq:N22-1} and \eqref{eq:N22-2} into \eqref{eq:v23a}
and using Young's inequality, we have
\begin{equation}
  \label{eq:v23b}
\sup_{t\in[0,T]}\|v_{23}(t)\|_{L^2_a}\lesssim
\left\{ \eta +\|v_0\|_{H^1}^{2/7}\|v_0\|_{L^2}^{6/7}
\left(\frac{\bM_2(T)^2}{\|v_0\|_{L^2}}\right)^{6/7}\right\} \bM_2(T)\,,
\end{equation}
where
$\eta=\|v_0\|_{H^1}^{1/2}\|v_0\|_{L^2}^{3/2}+\|v_0\|_{H^1}^{2/7}\|v_0\|_{L^2}^{6/7}
+\|v_0\|_{L^2}^{6/7}\bM_2(T)^{4/7}+\bM_2(T)^{16/7}$.
\par
On the other hand, applying Proposition~\ref{prop:l-smoothing} to $v_{23}(t)$,
we have
\begin{equation}
  \label{eq:v23c}
\|v_{23}(t)\|_{L^2(0,T;H^1_a)}\lesssim 
\left\{ \eta +\|v_0\|_{H^1}^{2/7}\|v_0\|_{L^2}^{6/7}
\left(\frac{\bM_2(T)^2}{\|v_0\|_{L^2}}\right)^{6/7}\right\} \bM_2(T)\,.
\end{equation}
Combining \eqref{eq:v21b}, \eqref{eq:v22a}, \eqref{eq:v22b},
\eqref{eq:v23b} and \eqref{eq:v23c}, we obtain 
\begin{equation*}
\bM_2(T)\lesssim \|v_0\|_{L^2}+\|v_0\|_{H^1}^{2/7}\|v_0\|_{L^2}^{6/7}
\left(\frac{\bM_2(T)^2}{\|v_0\|_{L^2}}\right)^{6/7}\bM_2(T)\,,
\end{equation*}
whence $\bM_2(T)\lesssim \|v_0\|_{L^2}$ if $\delta_5$ is sufficiently small.
Thus we complete the proof.
\end{proof}

\section{Proof of  Theorems~\ref{thm:1} and \ref{thm:2}}
\label{sec:thms}
Now we are in position to complete the proof of Theorems~\ref{thm:1}
and \ref{thm:2}.
\begin{proof}[Proof of Theorem~\ref{thm:1}]
By Lemma~\ref{lem:decomp}, the decomposition \eqref{eq:decomp2}
satisfying the orthogonality conditions \eqref{eq:orth1} and
\eqref{eq:orth2} exists on $[0,T]$ for a $T>0$.
Moreover, Lemma~\ref{lem:v2est-a} implies
$$\bM_{tot}(T) \lesssim \|v_0\|_{L^2}\le \frac12\min_{0\le i\le 5}\delta_i$$
if $\|v_0\|_{L^2}$ is sufficiently small. Hence it follows from
Lemma~\ref{lem:decomp} that the decomposition \eqref{eq:decomp2} satisfying 
\eqref{eq:orth1} and \eqref{eq:orth2} persists on $[0,\infty)$.
Thus we may take $T=\infty$ in Lemma~\ref{lem:modulation} and
it follows that
\begin{equation}
  \label{eq:bminfty}
\bM_{tot}(\infty)\lesssim \|v_0\|_{L^2}\,,
\end{equation}
\begin{equation}
  \label{eq:c-c+,pre1}
\sup_{t\ge0}(|c(t)-c_0|+|\dot{x}(t)-c(t)|)\lesssim \|v_0\|_{L^2}  
\end{equation} and
\begin{align*}
\|u(t,\cdot)-\varphi_{c_0}(\cdot-x(t))\|_{L^2}
\le  & \|\varphi_{c(t)}-\varphi_{c_0}\|_{L^2}+\|v(t,\cdot)\|_{L^2}
\\ \lesssim & |c(t)-c_0|+\|v(t,\cdot)\|_{L^2}\lesssim \|v_0\|_{L^2}^{1/2}\,.
\end{align*}
Thus we prove \eqref{eq:os-2}.
\par
Next we will prove \eqref{dotx} and \eqref{c-c+}.
By Corollary~\ref{cor:vir1},
\begin{equation}
  \label{eq:v1-lim}
\|v_1(t)\|_W^2 \lesssim \int_\R \chi_a(y)v_1^2(t,y)\,dy\to0\quad\text{as $t\to\infty$.}
\end{equation}
Integrating \eqref{eq:c-refine} with respect to $t$ and
combining the resulting equation with \eqref{eq:v1-lim} and the fact that
$$\int_0^\infty(\|v_1(t)\|_W^2+\|v_2(t)\|_{L^2_a}^2)\,dt
\lesssim \bM_1(\infty)^2+\bM_2(\infty)^2\lesssim \|v_0\|_{L^2}^2\,,$$
we see that $c_+:=\lim_{t\to\infty}c(t)$ exists and
\begin{equation}
  \label{eq:c-c+,pre2}
  |c_+-c_0|\lesssim \|v_0\|_{L^2}\,.
\end{equation}
Moreover, applying the H\"older inequality to \eqref{eq:v2-p1}
separately on the integral intervals $[0,t/2]$ and $[t/2,t]$
and using \eqref{eq:N1-w}--\eqref{eq:l-w}, we can show that
\begin{equation}
  \label{eq:v2-lim}
\lim_{t\to\infty}\|v_2(t)\|_{L^2_a}=0\,.
\end{equation}
By \eqref{eq:x,c-est}, \eqref{eq:v1-lim} and \eqref{eq:v2-lim},
we see that $\dot{x}(t)- c(t)\to0$ as $t\to\infty$.
Thus we prove \eqref{dotx}. Eq.~\eqref{c-c+} follows from \eqref{eq:c-c+,pre1}
and \eqref{eq:c-c+,pre2}.
\par
Finally, we will prove \eqref{eq:as-2}. 
Since $\lim_{t\to\infty}\|\varphi_{c(t)}-\varphi_{c_+}\|_{L^2}=0$,
it suffices to prove that as $t\to\infty$,
\begin{equation}
  \label{eq:asymp1}
\|u(t,\cdot)-\varphi_{c(t)}(\cdot-x(t))\|_{L^2(x\ge\sigma t)}
=\|v(t,\cdot)\|_{L^2(y\ge \sigma t-x(t))}\to0\,.
\end{equation}
Note that \eqref{eq:v1-lim} and \eqref{eq:v2-lim} already imply
\begin{equation}
  \label{eq:v-lim}
  \lim_{t\to\infty}\int_\R \chi_a(y)v^2(t,y)=0\,.
\end{equation}
We may write \eqref{eq:v} as
\begin{equation}
  \label{eq:v-2}
  \pd_tv+\pd_y^3v-\dot{x}\pd_yv+\ell+\pd_yf(v)+\pd_yN_3(t)=0\,,
\end{equation}
where $N_3(t)=f(\varphi_{c(t)}+v)-f(\varphi_{c(t)})-f(v)=6\varphi_{c(t)}v$.
\par

Let $c_1\in(0,\sigma)$, $t_1>0$ and $y_1(t)=c_1(t-t_1)-x(t)+x(t_1)$.
Multiplying \eqref{eq:v-2} by $2\chi_a(y-y_1(t))v(t,y)$ and integrating the
resulting equation by parts, we have 
\begin{equation}
  \label{eq:virial-v}
  \begin{split}
&\frac{d}{dt}\int_\R \chi_a(y-y_1(t))v^2(t,y)\,dy
+\int_\R \chi_a'(y-y_1(t))\{3(\pd_yv)^2+c_1v^2-g(v)\}(t,y)\,dy
\\ &=
\int_\R \chi_a'''(y-y_1(t))v^2(t,y)\,dy
+J(t)\,,    
  \end{split}
\end{equation}
where $g(v)=2f(v)v-2\int_0^vf(u)\,du$ and
$$J(t)=-2\int_\R\chi_a(y-y_1(t))v(t,y)(\ell(t)+\pd_yN_3(t))\,dy\,.$$
Lemma~\ref{lem:weight} implies
that
\begin{equation}
  \label{eq:pqf}
\int_\R \chi_a'(y-y_1(t))\{3(\pd_yv)^2+c_1v^2-g(v)\}(t,y)\,dy\ge
\int_\R \chi_a'''(y-y_1(t))v^2(t,y)\,dy\,.
\end{equation}
if $a$ and $\|v(t)\|_{L^2}\le\bM_v(\infty)$ is sufficiently small.
Note that $|\chi_a'''|\le 4a^2\chi_a'$.
\par
By \eqref{eq:x,c-est2} and the fact that $\ell$ and $N_3$ are exponentially
localized by $\varphi_{c(t)}$ and its derivatives, we have
\begin{equation}
  \label{eq:J1}
  |J(t)|\lesssim \|v_1(t)\|_{W_1}^2+\|v_2(t)\|_{H^1_a}^2\,,
\end{equation}
and $J\in L^1(0,\infty)$ by \eqref{eq:bminfty} and \eqref{eq:J1}.
Integrating \eqref{eq:virial-v} over $[t_1,t]$, we obtain
\begin{equation*}
  \int_\R \chi_a(y-y_1(t))v^2(t,y)\,dy \le \int_\R \chi_a(y)v^2(t_1,y)\,dy+\int_{t_1}^\infty J(s)\,ds\,.
\end{equation*}
Let $t\ge t_1\to\infty$. Then by \eqref{eq:v-lim} and the fact that $J\in L^1(0,\infty)$,
$$\lim_{t\to\infty}\int_\R \chi_a(y-y_1(t))v^2(t,y)\,dy=0\,.$$
Since $\sigma t-x(t)\ge y_1(t)$ for $t$ sufficiently larger than $t_1$, we conclude \eqref{eq:as-2}.
This completes the proof of Theorem~\ref{thm:1}.
\end{proof}
Since Theorem~\ref{thm:2} can be shown in exactly the same way as
the proof of Theorem~\ref{thm:1}, we omit the proof.

\section{Appendix: Weighted Sobolev inequalities}
\label{sec:ap1}
In this section, we recollect weighted Sobolev estimates.
To prove Lemma~\ref{lem:virial}, we use the following weighted inequality
as in \cite{MT}.
\begin{lmm}
\label{lem:weight}
Let $p=1$, $2$ or $3$ and $\eps>0$. Let $\chi_\eps(x)=1+\tanh\eps x$.
Then for every $v\in H^1(\R)$ and $x_0\in\R$, 
$$\left|\int_\R\chi_\eps'(x+x_0)v(x)^{p+1}\,dx\right|\le 
(1+2\eps)^{(p-1)/2}\|v\|_{L^2}^{p-1}\int_\R\chi_\eps'(x+x_0)
\left(v'(x)^2+v(x)^2\right)\,dx\,.$$
\end{lmm}
\begin{proof}
Since the case $p=1$ is obvious
and the case $p=2$ follows from the cases $p=1$ and $p=3$,
we only need to prove the case $p=3$.
Since $\lim_{x\to\pm\infty}v(x)=0$ for $v\in H^1(\R)$,
  \begin{align*}
\chi_\eps'(x+x_0)v(x)^2=& \int_{-\infty}^x \left(\chi_\eps'(y+x_0)v(y)^2\right)'\,dy
\\ = & \int_{-\infty}^x \chi_\eps''(y+x_0)\,v(y)^2\,dy
+2\int_{-\infty}^x \chi_\eps'(y+x_0)v(y)v'(y)\,dy\,.
  \end{align*}
Using the Schwarz inequality and the fact that
$ 0<\chi_\eps'(x)<2\eps\chi_\eps(x)$ and
$|\chi_\eps''(x)|\le2\eps\chi_\eps'(x)$ for every $x\in\R$, we have
$$\sup_{x\in\R}\chi_\eps'(x+x_0)v(x)^2
\le \int_\R \chi_\eps'(x+x_0)\left((1+2\eps)v(x)^2+v'(x)^2\right)\,dx\,.$$
Thus we have
\begin{align*}
\left|\int_\R \chi_\eps'(x+x_0)v^4(x)\,dx\right|
\le & \|v\|_{L^2}^2\sup_{x\in\R}\chi_\eps'(x+x_0)v(x)^2
\\ \le & \|v\|_{L^2}^2
\int_\R \chi_\eps'(x+x_0)\left((1+2\eps)v(x)^2+v'(x)^2\right)\,dx\,.
\end{align*}
Thus we complete the proof.
\end{proof}

\begin{lmm}
\label{cl:winfty}  Let $a>0$. Then
\begin{gather}
\label{eq:w1}
\|w^2\|_{L^\infty_a}\le 2\|w\|_{L^2}^\theta\|\pd_xw\|_{L^2}^{1-\theta}
\|\pd_xw\|_{L^2_a}^\theta\|w\|_{L^2_a}^{1-\theta}
\quad\text{for $\theta\in[0,1]$,}
\\ \label{eq:w2}
\|w\|_{L^\infty_a}^2\lesssim \|w\|_{L^2_a}\|w\|_{H^1_a}\,,\quad  
\|e^{-2a|\cdot|}w^2\|_{L^\infty} \lesssim \|w\|_W\|w\|_{W_1}\,.
\end{gather}
\end{lmm}
\begin{proof}[Proof of Lemma~\ref{cl:winfty}]
It suffices to prove Lemma~\ref{cl:winfty} for $w\in C_0^\infty(\R)$.
Since $e^{ax}$ is monotone increasing,
$$e^{ax}w^2(x)=-2e^{ax}\int_x^\infty w(y)w'(y)\,dy
\le 2 \int_x^\infty e^{ay}|w(y)||w'(y)|\,dy\,.$$
By the Schwarz inequality, we have $$e^{ax}w^2(x)\le2\|w\|_{L^2}\|w'\|_{L^2_a}
\,,\quad e^{ax}w^2(x)\le 2\|w\|_{L^2_a}\|w'\|_{L^2}
\quad\text{for any $x\in\R$.}$$  
Interpolating the above inequalities, we have \eqref{eq:w1}.
We can prove \eqref{eq:w2} in the similar way.
\end{proof}

\end{document}